\documentclass[12pt]{article}
\usepackage{amssymb,dsfont}
\usepackage{amsmath,amsthm}

\usepackage{rotating}

\usepackage{makeidx}
\makeindex

\newcommand{\R}{\mathbb{R}}
\newcommand{\E}{\mathbb{E}}
\renewcommand{\P}{\mathbb{P}}
\newcommand{\N}{\mathbb{N}}

\newcommand{\Z}{\mathbb{Z}}

\newcommand{\LO}{\mathcal{L}}
\newcommand{\D}{\mathcal{D}}
\newcommand{\vv}{{\vskip 1mm \noindent}}

\newdimen\trimheight\trimheight9.71truein   
\newdimen\trimwidth \trimwidth6.48truein     
\newdimen\typeheight\typeheight8in   
\newdimen\typewidth \typewidth5in
\newdimen\draftrule \draftrule=0pt
\newdimen\tempdimen
\newdimen\tablewidth
\newdimen\normaltextheight
\newbox\tempbox
\newdimen\tablewd 

\voffset= - 0.8 cm
\usepackage{color}
\numberwithin{equation}{section}
\newtheorem*{theorem*}{Theorem}
\newtheorem*{corollary*}{Corollary}
\newtheorem{theorem}{Theorem}[section]
\newtheorem{lemma}[theorem]{Lemma}
\newtheorem{proposition}[theorem]{Proposition}
\newtheorem{corollary}[theorem]{Corollary}
\theoremstyle{definition}
\newtheorem{definition}[theorem]{Definition}

\newtheorem{remark}[theorem]{Remark}

\date{}

\begin{document}

\title { \bf{
On the Cauchy problem for non-local Ornstein--Uhlenbeck operators
}
}

\author{ \bf{E. Priola} \\
{Dipartimento di Matematica} \\ Universit\`a degli Studi di Torino
\\ enrico.priola@unito.it
\\ \\
\bf{S. Trac\`a} \\
 Operations Research Center\\ Massachusetts Institute of Technology
\\stet@mit.edu}

 \maketitle
%
\noindent {\bf Abstract:} We   study
 the Cauchy problem involving non-local 
Ornstein-Uhlenbeck operators in finite and infinite dimensions.
We  prove  classical solvability  
 without requiring 
that the L\'evy   measure 
corresponding to the large jumps part
has a first finite moment.
%
%
Moreover, we determine    
 a core  of regular functions which is invariant 
for the
associated 
transition Markov  semigroup.
Such a core 
allows to characterize the marginal laws 
of the Ornstein-Uhlenbeck 
stochastic process  as unique
solutions to Fokker-Planck-Kolmogorov  equations for measures.




\section{Introduction and notation
}
In this paper we investigate  solvability of the Cauchy problem involving  non-local Ornstein-Uhlenbeck operators both in   finite and infinite dimensions. We  also determine a core of regular functions which is invariant for the transition Ornstein-Uhlenbeck semigroup. 
 Differently with respect to  
 recent papers (see  \cite[Section 5]{Applebaumart1}, \cite[Section 4.1]{Kn} and \cite[Section 2]{W}) 
to study the core problem
we do not require that the associated L\'evy   measure $\nu$ 
 corresponding to the large jumps part
has a first finite moment (see \eqref{add1}). 


Let us first introduce the Ornstein-Uhlenbeck operator ${\cal L}_0$ in $\R^d$ and  its associated stochastic process.
The  operator ${\cal L}_0$ is defined 
as
\begin{align} \label{ouin}
{\cal L}_0 f(x)= \frac{1}{2}\sum_{j,k =1}^d Q_{jk} \, \partial_{x_j x_k}^2 f(x)
 +\sum_{j=1}^d a_j \partial_{x_j} f (x) +  \sum_{j,k =1}^d A_{jk} \, x_k  
\partial_{x_j } f(x) 
 \\   \nonumber    + \int_{\R^d} \Big( f(x+y) -f(x) - \mathds{1}_{\{ |y| \le 1\}}\, (y) \sum_{j=1}^d 
y_j\,{\partial_{x_j} f}(x) \Big)\nu(\mbox{d}y), \;\;\ x \in \R^d,
\end{align}
 where $\mathds{1}_{\{ |y| \le 1\}}$ is the indicator function of the closed ball with center $0$ and radius $1$,   $Q = (Q_{ij})$ and $A = (A_{ij})$ are given $d \times d$ real matrices ($Q$ being   symmetric and non-negative definite). Moreover $a = (a_1, \ldots, a_d) \in \R^d$ and 
$\nu$ is a  L\'evy jump measure, i.e., $\nu$ is a $\sigma$-finite Borel measure on $\R^d$ such that  
\begin{equation} \label{nuu}
\text{$\nu (\{0\})=0$ and} \;\; \int_{\R^d} (1 \wedge |y|^2) \, \nu(dy) < \infty
\end{equation}
($a \wedge b$ indicates the minimum between $a$ and $b \in \R$).
The function $f: \R^d \to \R$ belongs to $C^2_b(\R^d)$ (i.e., $f$ is  bounded and continuous together with its first and second partial derivatives) and the integral in \eqref{ouin} is well defined thanks to the Taylor formula.
 The associated Ornstein-Uhlenbeck
process (OU process)  solves the 
following   SDE
driven by a L\'evy process $Z$:
\begin{equation}\label{ou}
  \left\{ \begin{array}{ll}
         \mbox{d}X_t =AX_t\mbox{d}t+ \mbox{d}Z_t, & t \geq0 \\
         X_0=x  ,  & x\in\R^d
                \end{array}\right.
\end{equation}
(see, for instance, \cite{SY}, 
\cite{SWYY} and \cite{Masuda}).
The matrix $A$ is the same as  in \eqref{ouin} and $Z= (Z_t)_{t \geq 0}$ $= (Z_t)$ is a $d$-dimensional  L\'evy
 process uniquely  determined in law by the previous  $Q$, $a$ and $\nu$   (cf. \cite[Section 9]{Sato}).
Ornstein-Uhlenbeck processes  with jumps $X = (X_t)$ 
$= (X_t^x)$ 
have several applications to Mathematical Finance and Physics
(see for instance, \cite{BNS}, \cite{CT} and \cite{GO00}).
The corresponding transition Markov semigroup $(P_t) = $ $(P_t)_{t\geq 0}$ is called the Ornstein-Uhlenbeck semigroup (or Mehler semigroup):
\begin{equation}
\label{ou38}
 P_t f(x) = \E [f(X_t^x)], \;\;\; t \ge 0,
\end{equation} 
for any $f: \R^d \to \R$ which is Borel and bounded (see also \eqref{for1}).
In Section 3.1  we prove 
well-posedness of the  Cauchy problem
\begin{equation}
\label{ca4}
\begin{cases} \partial_t u (t,x) = \LO_0 u(t, x)
\\
u(0,x) = f(x),\;\;\; x \in \R^d, \;\; f \in C^2_b(\R^d),
\end{cases}
\end{equation}
where $\LO_0 u(t, x) = (\LO_0u(t, \cdot))(x)$ (see Theorem \ref{PPCAUCHY}).
We show that there exists a  unique bounded classical  solution  given by $u(t,x)= P_t f (x)$, $t \ge 0, $ $x \in \R^d$.
 Our result is not covered by regularity results on singular  pseudodifferential operators (cf. \cite{Ko89}). Moreover, it can not be deduced by perturbation arguments
using known   results  for the Ornstein-Ulenbeck semigroup (see, in particular, \cite{SY} and \cite{Masuda}). 
To prove solvability of \eqref{ca4}  we first establish the crucial formula 
\begin{equation}\label{Peq1}
P_t f(x)= f(x) + \int_0^t \LO_0 (P_s f ) (x) \mbox{d}s, \,\,\,\,t\geq0, 
x\in\R^d,\;\; f \in C^2_b(\R^d).
\end{equation}
(see Theorem \ref{PCAUCHY}). 
  Note that in 
\cite[Theorem 
3.1]{SY} it is  proved that 
\begin{equation}\label{ya1}
P_t f(x)= f(x) + \int_0^t P_s (\LO_0 f ) (x) \mbox{d}s, \,\,\,\,t\geq0, 
x\in\R^d , \;\; f \in C^2_K(\R^d).
\end{equation}
(we write $f \in C^2_K(\R^d) $ if $f \in C^2_b(\R^d)$ and $f$ has compact support).
Even   assuming  $f \in C^2_K(\R^d)$, formula \eqref{Peq1} can not be obtained directly from 
\eqref{ya1}
since 
the space $C_K^2(\R^d)$ is not 
invariant for the OU  semigroup $(P_t)$ (cf. Remark \ref{remimp}).
On the other hand \eqref{ya1} does not hold in general for $f \in C^2_b(\R^d)$
since $\LO_0 f$ can  grow linearly and so  
$P_t (\LO_0 f ) (x)$ 
could be  not well-defined 
without requiring the additional assumption (cf.  \eqref{uno1})
\begin{equation}
\label{add1}
\int_{ \{ |y| >1\}} |y| \nu (dy) < \infty.
\end{equation}
It is well-known that the OU semigroup $(P_t)$ is a $C_0$-semigroup (or strongly continuous semigroup) of 
contractions on $C_0(\R^d)$ (the Banach space of all real continuous functions on $\R^d$ which vanish at infinity, endowed with the supremum norm); see \cite{SY}, \cite{Masuda} and \cite{Tr} for a more direct proof. Let us denote by 
${\cal L}$ its generator.
Using \eqref{Peq1} in Theorem \ref{cor1} we  
show  that 
 \begin{equation} \label{doo}
 \D_0 = \Big \{ f \in C_0^2 (\R^d)  \; \; 
\text{such that } \;\;
\sum_{j,k =1}^d A_{jk} \, x_k \, 
\partial_{x_j } f \in
C_0(\R^d)  \Big\}
 \end{equation}
is {\it  invariant}  for the OU semigroup  ($f \in C^2_0(\R^d)$ if $f \in C^2_b(\R^d)$ and $f$ and its first and second partial derivatives belong to $C_0(\R^d))$.  
Note that 
this property 
implies that  the mapping:  $x \mapsto Ax \cdot DP_tf(x)$ is bounded on $\R^d$  when $f \in \D_0$ without assuming 
\eqref{add1}.
It turns out that $\D_0$  is 
also a core for ${\cal L}$  and ${\cal L} f= {\cal L}_0 f$, $f \in \D_0$.
Clearly, 
if  $f\in \D_0$ then  both \eqref{Peq1} and  \eqref{ya1} hold
 (see Corollary \ref{co1}).  

Starting from Section 4, we extend  the main results of Section 3 
to infinite dimensions,
replacing $\R^d$ with a given
real separable  Hilbert space ${H}$. Infinite dimensional Ornstein-Uhlenbeck processes  are solutions of linear 
stochastic evolution equations and  are formally  similar to  \eqref{ou}; we   assume  that  $A$ is the generator of a $C_0$-semigroup on $H$ and $Z $ is an $H$-valued L\'evy processes. Such processes 
allow to solve basic linear SPDEs 
(cf. \cite{DZ}, \cite{D}, \cite{PZ} and the references therein).
 Ornstein-Uhlenbeck processes with jumps in infinite dimensions were first 
studied in \cite{Ch}. A more  general approach to such 
processes 
using generalised Mehler semigroups has been initiated in 
\cite{BRS} (see also  \cite{FR}, \cite{LR}, \cite{PZ1}, \cite{Kn}, \cite{W} and Remark \ref{mehler}). 

 In Theorem \ref{PCAUCHYi} we 
extend formula
\eqref{Peq1} to  infinite dimensions when $f \in C^2_b(H)$ and $x \in D(A)$ (i.e., $x$ belongs to the domain of $A$). We use such formula to  show existence and uniqueness of solutions for  an infinite-dimensional Cauchy problem like \eqref{ca4} when $\R^d$ is replaced by $H$ (see Theorem \ref{PPCAUCHY2}); we assume that the  initial datum $f$ belongs to  $C^2_b(H)$ and  that a compatibility condition between $Df(x) $ and $A$ is satisfied (see the definition of the space $C^2_A(H)$ in \eqref{ca1}). 
This result of well-posedness seems to be new even for local infinite dimensional 
OU operators corresponding to the case when $Z$ is a Wiener process.

To study an infinite dimensional OU semigroup $(P_t)$
it is natural to consider it as acting in  $C_b(H)$ or $UC_b(H)$  which are both invariant for the semigroup. Here 
$C_b(H)$ 
(resp. $UC_b(H)$) consists  of all real bounded and continuous (resp.  uniformly continuous) functions on $H$.
Indeed  $C_0(H)$ which generalizes $C_0(\R^d)$ is invariant for $(P_t)$ only under  quite restrictive assumptions (see 
 page 91 of \cite{Applebaumart1}). On the other hand, it is well known that $(P_t) $ is  not strongly continuous neither on  $C_b(H)$ nor in  $UC_b(H)$ if we consider the sup-norm topology (see \cite{Ce}, \cite{CG}, \cite{P}, \cite{GK}, \cite{Ku} where possible  approaches to study  Markov semigroups in $C_b(H)$  or $UC_b(H)$ are proposed). 
In $C_b(H)$ one can define an infinitesimal generator ${\cal L}$  in a pointwise sense (see \eqref{gene}) or in other equivalent ways (see Remark \ref{d33}). This generator coincides with the one investigated in \cite{Applebaumart1}.

 In Section 5.2  we determine two natural   pointwise cores  ($\pi$-cores) ${\tilde \D_0}$  and $\D_1$ for $\cal L$. They are both invariant for  $(P_t)$ and further the restriction of $\cal L$ to ${\tilde \D_0}$  and $\D_1$
 coincides with ${\cal L}_0$. The $\pi-$core $\tilde \D_0$ is a kind of infinite-dimensional generalization of $\D_0$ given in \eqref{doo}. 
 On the other hand $\D_1$ is  similar to the space introduced in  \cite{Ma} when the L\'evy noise is a Wiener process. The definition of $\D_1$ is a bit involved but this space can also be used  to study  generalised Mehler semigroups (cf. Remark \ref{mehler}). We discuss in Remark \ref{app} another possible core 
used in \cite{Applebaumart1} under the assumption
$\int_{ \{ |x| >1\}} |x| \nu (dx) < \infty$ (see also \cite{GK}).

We mention that  cores of regular bounded functions are also useful  to investigate the Ornstein-Uhlenbeck semigroups in $L^p(\mu)$ with respect to an invariant measure $\mu$ assuming that such a measure exists  (see \cite{Ch}, \cite{LR}, \cite{Kn} and the references therein). 
Moreover, such cores allow to study 
Fokker-Planck-Kolmogorov equations for measures (see \cite{BDR04}, \cite{Ma}, \cite{BDR}, \cite{Kn}, \cite{W} and the references therein). On this respect, in Section 6, following \cite{Ma}, we  
show that both $\tilde \D_0$ and $\D_1$  can be used 
to  prove that the marginal laws of Ornstein-Uhlenbeck processes are the  unique solutions to 
Fokker-Planck-Kolmogorov equations for measures $(\gamma_t)$ , i.e.,
\begin{equation*}
\begin{cases}
\frac{d}{dt} \int_H f(x) \gamma_t(dx) = 
\int_H {\cal L}_0 f(x) \gamma_t(dx), \;\;\; f \in {\D}, \; t \ge 0,
\\
 \gamma_0 = \delta_x,
\end{cases}  
\end{equation*}
$x \in H$; here the space $\D$ can be $\tilde \D_0$ or $\D_1$ (see Theorem \ref{cio}).

\medskip
\noindent {\bf Notation.} In $\R^d$ 
by $\langle \cdot, \cdot \rangle$  and $|\cdot |$ we indicate 
the usual inner
product and the Euclidean norm, $d \ge 1$. By
$B_b(\R^d, \R^k)$, $d,k \ge 1$, we denote  the Banach space of  all Borel and 
bounded functions $f$ 
from $\R^d$ into $\R^k$  endowed with norm $\Vert f \Vert_{0}=
\sup_{x\in\R^d}\vert f(x) \vert$. When $\R^k = \R$ we set $B_b(\R^d, \R)
= B_b(\R^d)$; this convention will be used for other function spaces as well. $C_b(\R^d, \R^k)$ is the closed subspace of $B_b(\R^d, \R^k)$ of all  
bounded and continuous functions.  Moreover, $UC_b(\R^d, \R^k)$ is the closed 
subspace of all bounded and  uniformly  continuous functions. 
We say  that $f \in C_0(\R^d, \R^k)$ if $f \in C_b (\R^d, \R^k)$ and 
$f$ vanishes at infinity (i.e., for any $\epsilon >0$ there exists a bounded 
set $A \subset \R^d$ such that $|f(x)| < \epsilon$ if $x \in \R^d \setminus A$).

For each integer $n \ge 1,$   we say that
 $g\in C_{b}^{n}({\mathbb{R}}^{d})$ if $ g \in C_b(\R^d)$ and $g$ is $n$-times 
 (Fr\'echet) differentiable on $\R^d$ with 
all the (Fr\'echet) derivatives $D^{j}g$ which are continuous and 
bounded on $\R^d$, $j =1, \ldots, n$.
We use $C^2_0(\R^d)$ to denote the space of functions $f\in C_b^2(\R^d) \cap 
 C_0(\R^d)$ 
such that  $D f$ and $D^2 f$ 
vanish at infinity. Finally, $C^2_K (\R^d) $ denotes the space of functions $f 
\in C^2_0(\R^d)$ with compact support.

Given a 
real separable Hilbert space $H$ and  a linear bounded operator
$T : H \to H$, we denote by $\| T\|_{L}$ its operator norm. If in addition  $T$  is 
 an Hilbert-Schmidt operator then 
$$
 \| T \|_{HS} = \Big(\sum_{k \ge 1} |T e_k|^2 \Big)^{1/2}  
$$
denotes its Hilbert-Schmidt norm (here
$(e_k)$ is an orthonormal basis in $H$ and $\langle \cdot , \cdot \rangle$, $|\cdot |$ denote the inner product and the  
norm in $H$). 
We indicate by $L_2(H)$ the  space consisting of all 
   Hilbert-Schmidt operators from $H$ into $H$; it is 
a real separable Hilbert 
space endowed  with the inner product: $T \cdot S = \text{Tr} (T^* S )$ 
$=\sum_{k \ge 1} \langle T e_k, Se_k \rangle$, $S, T \in L_2(H)$
 (we refer to  
 Appendix C in \cite{DZ} 
for more details). 

The previous function spaces can be easily generalised  when $\R^d$ is 
replaced by $H$, i.e.,  we can consider  the spaces 
$B_b(H)$, 
$C_b(H)$, $UC_b(H)$, $C_0 (H)$, $C^k_b(H)$, $k = 1,2$. In particular,
$f \in C^1_b (H)$ if $f: H \to \R$ is Fr\'echet differentiable in $H$ and $f: H \to \R$, $Df : H \to H$ are bounded and continuous (we write $Df \in C_b(H,H)$). Moreover, $f \in C^2_b (H)$ if $f \in C^1_b(H)$, $f: H \to \R$ is twice Fr\'echet differentiable in $H$ with  $D^2f(x) \in L_2(H)$, $x \in H$, and $D^2f : H \to L_2(H)$ is  bounded and continuous.



\section{Preliminaries}
\label{Basic definitions and results}

Here we review basic facts on L\'evy processes and introduce  the 
Ornstein-Uhlenbeck process with values in
$\R^d$. We refer to \cite{Sato}, \cite{Applebaum} and \cite{SY} for more 
details.

We fix a
  stochastic basis
  $(\Omega, {\cal F}, ({\cal F}_t)_{t \ge 0}, \P)$ which satisfies
  the usual assumptions (see, for instance,
  page 72 in  \cite{Applebaum}). An
  $({\cal F}_t)$-adapted
  $d$-dimensional stochastic process $Z=(Z_t)$ $= (Z_t)_{t \ge 0}$,  $d \ge 1$, 
is a  {\it L\'evy  process}
  if it is  continuous in probability,
  it  has  stationary increments,
     c\`adl\`ag  trajectories, $Z_t - Z_s$ is independent of
     ${\cal F}_s$, $0 \le s \le t$,
      and $Z_0=0$.  
Recall that there exists a unique $\psi : \R^d \to {\mathbb C}$ such that 
 $$
 \E [e^{i \langle u, Z_t\rangle}] = e^{- t
\psi(u)},\, u \in \R^d, \; t \ge 0;
$$
$\psi$ is called the {\it exponent} (or symbol)
 of $Z $ ($\E$ denotes expectation with respect to $\P$) . The \emph{L\'evy-Khintchine representation} for $\psi$  is 
\begin{equation} \label{ft1}
 \psi(u)= \frac{1}{2}\langle Q u,u  \rangle   - i \langle a, u\rangle
- \int_{\R^d} \Big(  e^{i \langle u,y \rangle }  - 1 - \, { i \langle u,y
\rangle} \, \mathds{1}_{\{ |y| \le 1\}} \, (y) \Big ) \nu (dy), 
\end{equation}
 $u \in \R^d,$ where 
   $Q$ is a symmetric $d \times d$ non-negative definite matrix, $a
\in
 \R^d$ (if $B \subset \R^d$, $\mathds{1}_B (x) =1$ if $x \in B$ and $\mathds{1}_B (x) =0$ if $x \not \in B$). Moreover,
  $\nu$
 is the {\it L\'evy (jump) measure} (or intensity measure) of $Z.$ Thus,  
$\nu$ is a $\sigma$-finite (Borel)
measure on $\R^d$, such that  \eqref{nuu} holds.
Note that 
\begin{equation}
\label{uno1}
\E [ |Z_t| ]< \infty, \;\;\; t \ge 0, \;\; \text{if and only if } \;\;
\int_{ \{ |y| >1\}} |y| \nu (dy) < \infty
\end{equation} 
 (see  \cite[Theorem 2.5.2]{Applebaum}).
The 
  \emph{ Ornstein-Uhlenbeck process} which solves \eqref{ou} is given by 
\begin{equation}\label{2}
 X^x_t  =  e^{tA} x + \int_0^t e^{(t-s)A}  dZ_s = e^{tA} x + Y_t,
  \;\; t \ge 0, \; x \in \R^d,
\end{equation}
where $e^{tA} = \sum_{k=0}^{\infty}\frac{t^k A^k}{k!} $  and the stochastic 
convolution $Y_t$ can be defined as a limit
in probability of suitable Riemann  sums (cf. page 104 in \cite{Sato} and \cite{Ch}). Let us denote by $\mu_t$ the law of 
$Y_t$.
The {law} $\mu_t^x $ of $X_t^x $ has
 {characteristic function} (or Fourier transform) $\hat {\mu_t^x}$ given by
\begin{equation}\label{c}
\hat {\mu_t^x} (h) 
= \E [e^{i \langle  X_t^x, h\rangle}]
=
e^{i \langle e^{tA}  x, h\rangle} \hat \mu_t (h) =
e^{i \langle e^{tA^*}  h, x\rangle}
 \,  \exp { \Big(- \int_0 ^t \psi (  e^{s A^*}  h) ds \Big)}, 
\end{equation}
 $h \in
 \R^d,$
where $A^*$ denotes the adjoint matrix of $A$ and $\psi$ is the exponent 
 of  $Z$ (cf.   \cite[Proposition 2.1]{Masuda}; clearly, $\mu_t = \mu_t^0$). Next we recall basic facts about OU semigroups $(P_t)$.
We define for $f \in B_b(\R^d),$ $x \in \R^d,$ $t \ge 0$
\begin{equation} \label{for1}
P_t f(x) = (P_t f)(x) = \E [f(X_t^x)]= \int_{\R^d} f(e^{tA} x + y) \mu_t(dy).
\end{equation} 
An important property is that, for any $f \in C_b(\R^d)$,  the mapping: 
\begin{equation} \label{con2}
(t,x) \mapsto P_tf(x) \; \text{ is continuous on } \;
 [0, +\infty) \times \R^d \; \;\text {and }
\end{equation} 
\begin{equation} \label{con3}
 \lim_{t \to 0^+ } P_t f = f, \;\;\; \text{uniformly on compact sets of $\R^d$}
\end{equation}
 (we refer to   \cite[Lemma 2.1]{BRS} which contains  a more general result;
see also    
 \cite[Section 4]{FR} and \cite[Theorem 4.1]{Applebaumart1}). Note that \eqref{con2} 
implies \eqref{con3}.

The spaces $B_b(\R^d)$, $C_b(\R^d)$, $UC_b(\R^d)$ and $C_0(\R^d)$ are all 
invariant for the 
OU semigroup (for instance, $P_t (C_b(\R^d)) \subset C_b(\R^d)$, $t \ge0$).
Moreover,  $(P_t)$ is a $C_0$-semigroup of 
contractions on $C_0(\R^d)$ (see \cite{SY}, \cite{Masuda} and \cite{Tr}), i.e.,
\begin{equation}
\label{strong1}
\lim_{t \to 0^+ } \| P_t f -f\|_{0} =0, \;\;\; f \in C_0(\R^d). 
\end{equation}


\section{ Ornstein-Uhlenbeck operators in $\R^d$  }

\subsection{Classical solvability of the Cauchy problem}

We show   well-posedness  of the Cauchy 
problem \eqref{ca4} involving  
${\cal L}_0$. We write
\begin{gather}
\nonumber {\cal L}_0 f(x) =    \langle A x ,  Df(x) \rangle  + {\cal L}_1 
f(x), \;\; \text{where} 
\\ \label{vi2}
  {\cal L}_1 f(x)   =  
\int_{\R^d}\left(  f(x+y) - f(x) 
-\mathds{1}_{\{\vert y \vert \leq 1\}} 
\langle y , Df(x) \rangle\right)\nu(\mbox{d}y)
\\ \nonumber
+  \frac{1}{2}\mbox{Tr}(QD^2f(x)) + \langle a, Df(x)\rangle, \;\;\; f \in C^2_b 
(\R^d),\; x \in \R^d.
\end{gather}
A {\it bounded classical solution $u$ } to the Cauchy problem \eqref{ca4}
 is a bounded and continuous real function defined on $E = [0,+\infty) \times \R^d$,  such that

\vv (i)
there exist classical partial derivatives 
  $\partial_{x_i} u$ and 
 $\partial_{x_i x_j}^2 u$, $i, j =1, \ldots, d$,
 which are bounded and continuous on $E$;

\vv (ii) $u(\cdot ,x)$ is a $C^1$-function on $[0,+\infty)$, $x \in \R^d,$ and $u$ solves \eqref{ca4}.


\medskip
Existence of classical solutions is based on the following result.
\begin{theorem}\label{PCAUCHY}
 Let $f\in C_b^2(\R^d)$. Then, for any $x \in \R^d$, the mapping: $t \mapsto 
\LO_0 (P_t f ) (x)$ is continuous on $[0, + \infty)$ and $\lim_{t \to 0^+}
  \LO_0 (P_t f ) (x) = \LO_0 f (x)$. Moreover,
\begin{equation}\label{Peq5}
P_t f(x)= f(x) + \int_0^t \LO_0 (P_s f ) (x) \mbox{d}s, \,\,\,\,t\geq0, 
x\in\R^d.
\end{equation}
\end{theorem}
In order to prove this result  we introduce the
    linear span $V(\R^d)$ of 
the real and imaginary parts
of the functions
 $x \mapsto e^{i \langle x, h \rangle}$, $h \in \R^d$.

We need an  approximation result with functions in $V(\R^d)$.   This is similar to \cite[Proposition 2.67]{D} and  \cite[Proposition 4.2]{Ma} 
(the main difference is that  these results do not consider approximations of second derivatives). A detailed proof is given in \cite{Tr}. We give a sketch of proof in Appendix.  

\begin{lemma}\label{Plemma}
Let $f\in C^2_b(\R^d)$. There exist a double sequence  
$(f_{nm})_{n,m\in\N} \subset V(\R^d)$ and a sequence $(f_n)_{n\in\N} \subset  
C^2_{b}(\R^d)$ such that, for any $n \ge 1,$  
$$ \lim_{m\rightarrow\infty}  \| f_{nm} - f_n\|_{0} =0, \;\;\;
 \lim_{m\rightarrow\infty}  \| Df_{nm} - Df_n\|_{0} =0,
\;\; \lim_{m\rightarrow\infty}  \| D^2f_{nm} - D^2f_n\|_{0} =0.$$
Moreover, there exists $M= M (f)>0$   such that, for any $n , m \ge 1,$
\begin{equation}\label{M}
 \| f_{nm} \|_{0} + \| D f_{nm} \|_{0} +
  \| D^2f_{nm} \|_{0}+
 \Vert f_{n} \Vert_{0} + \Vert Df_{n} \Vert_{0} + \Vert D^2 
f_{n}\Vert_{0} \leq M \end{equation}
and, for any $x \in \R^d,$
$$
\lim_{n \rightarrow \infty} f_{n}(x)= f(x),\;\; 
\lim_{n \rightarrow \infty} Df_{n}(x)= Df(x),
\;\; \lim_{n \rightarrow \infty} D^2f_{n}(x)= D^2f(x).
$$
\end{lemma}
\begin{proof} [Proof of Theorem \ref{PCAUCHY}]
The 
first 
assertion about the continuity of $t \mapsto  
\LO_0 (P_t f ) (x)$ follows easily using property \eqref{con2} together 
 with  the 
following identity
\begin{align} \label{dd2}
\LO_0 (P_t f ) (x) =& \langle e^{tA} A x , P_t Df (x) \rangle + \\
\nonumber +& \int_{\R^d}\left(  P_tf(x+y) - P_t f(x) -\mathds{1}_{\{\vert y 
\vert \leq 1\}} \langle e^{tA}y , P_t Df(x) \rangle\right)\nu(\mbox{d}y) \; +\\ 
\nonumber +&  \frac{1}{2}\mbox{Tr}(Qe^{tA^*}P_t D^2f (x)e^{tA}) + \langle 
e^{tA} a, P_t 
Df(x)\rangle
\end{align}
 (note that   $P_t Df(x) = \E [Df(X^x_t)]$ and $P_t D^2f(x) = \E 
[D^2f(X^x_t)]$).

We will  first prove \eqref{Peq5} for $f\in V(\R^d)$. Then  when $f \in 
C^2_b 
(\R^d)$ we will use an approximating argument which is based on the previous 
lemma. We split the proof into two parts.

\vv \text{\it I Step.}
We  show \eqref{Peq5} for $f\in V(\R^d)$. It suffices to check that \eqref{Peq5} holds when
$$f(x)= 
e^{i \langle h , x \rangle}, 
$$ where $h\in\R^d$.
When $x\in\R^d$ is fixed, we have  (see \eqref{c})
\begin{align*}
&\frac{\partial}{\partial t} (P_t f)(x) \bigg\vert_{t=0} 
=\frac{\partial}{\partial t} (\E[ e^{i\langle h , e^{tA}x + \int_0^t e^{(t-s)A} 
dZ_s \rangle}])\bigg\vert_{t=0}=\\
&=i \langle  h , e^{tA} Ax \rangle e^{i\langle h , e^{tA}x \rangle} 
e^{-\int_0^t 
\psi (e^{sA^*} h) ds } - e^{i \langle h ,  e^{tA} x \rangle } \psi ( e^{tA^*} h) 
 e^{-\int_0^t \psi ( e^{sA^*} h) ds  } \bigg\vert_{t=0} =\\
&=  \langle  D f(x) , A x \rangle - e^{i \langle h , x \rangle} \psi(h),
\end{align*}
where $\psi$ is the exponent of the L\'evy process $(Z_t)$.  
Using the L\'evy-Khintchine formula, we have that (cf. \eqref{vi2})
\begin{equation*}
- e^{i \langle h , x \rangle} \psi( h) = \LO_1(e^{i\langle h , \cdot 
\rangle})(x)=(\LO_1 f)(x),
\end{equation*}
and therefore,
\begin{equation}
\frac{\partial}{\partial t} (P_t f)(x) \bigg\vert_{t=0} = \LO_0 f(x).
\end{equation}
Similarly, recalling that $Y_t=\int_0^t e^{(t-s)A} \mbox{d}Z_s$, we can compute 
the derivative for $t>0$:
\begin{align} \label{new1}
&\frac{\partial}{\partial t} (P_t f)(x) \bigg\vert_{t>0} 
=\frac{\partial}{\partial t} (e^{i\langle h , e^{tA}x \rangle}\E[e^{i\langle h , 
Y_t \rangle}])\bigg\vert_{t>0}=\\ \nonumber
&=i \langle  h , e^{tA} A x \rangle e^{i\langle h , e^{tA}x 
\rangle} e^{\int_0^t 
\psi (e^{sA^*} h) ds } + e^{i \langle h ,  e^{tA} x \rangle } \psi (e^{tA^*} h)  
e^{\int_0^t \psi (e^{sA^*} h) ds},
\end{align}
using that $e^{i \langle h ,  e^{tA} x \rangle} = e^{i  \langle e^{tA^*} h ,   x 
\rangle }$, we find that
 $
\frac{\partial}{\partial t} (P_t f)(x) = \LO_0 (P_t f)(x),  $ $\,t\geq0.$ 
Integrating with respect to $t$, we get the assertion.

\vv \text{\it II Step.} We prove \eqref{Peq5} when  $f \in C^2_b (\R^d)$. We choose an approximating  sequence 
$(f_{nm})_{n,m\in\N} \subset V (\R^d)$ as  in Lemma \ref{Plemma}.
We can write
\begin{equation}\label{Peqfnm}
P_t f_{nm}(x)= f_{nm}(x) + \int_0^t \LO_0 (P_s f_{nm}) (x) \mbox{d}s, 
\,\,\,\,t\geq0, x\in\R^d,
\end{equation}
for any $n,m\in\N$.
In order to pass to the limit in \eqref{Peqfnm}
  we fix $T>0$, $x \in \R^d$, and study the convergence of $\LO_0(P_s f_{nm})(x)$, 
with $s \in [0,T]$.  The term
\[
\int_{\R^d} (P_sf_{nm}(x+y) - P_sf_{nm}(x) - \mathds{1}_{\{ |y|\le 1 \}}\langle 
y , DP_sf_{nm}(x)\rangle ) \nu (\mbox{d}y)
\]
can be written as  $I_{nm}(x) + J_{nm}(x)$, where
\begin{equation} \label{new2}
 I_{nm}(x) =\int_{\{\vert y \vert > 1 \}}  (P_sf_{nm}(x+y) - 
P_sf_{nm}(x)) 
\nu (\mbox{d}y) \;\;\text{and}
\end{equation}
$$
J_{nm}(x) = \int_{\{\vert y \vert \le 1 \}}  (P_sf_{nm}(x+y) - P_sf_{nm}(x) - 
\langle y , DP_sf_{nm}(x)\rangle) \nu (\mbox{d}y).
$$
Passing to the limit first as $m \to \infty$ and then as $n \to \infty$, 
we get 
\begin{equation}
\label{new3}
\lim_{n \to \infty } (\lim_{m \to \infty} I_{nm}(x)) = \int_{\{\vert y \vert 
> 1 \}}  (P_sf(x+y) - P_sf(x)) \nu (\mbox{d}y).
\end{equation}
(recall that $\nu$ is a 
finite measure on $\{ |y| > 1 \}$).
To deal with  $J_{nm}(x)$ we first observe that, using Taylor formula,
\[
\vert P_sf_{nm}(x+y) - P_sf_{nm}(x) - \langle y , DP_sf_{nm}(x)\rangle \vert 
\leq \vert y \vert^2 \Vert D^2 P_s f_{nm} \Vert_{0} \leq
\]
\[ \le C_T \vert y \vert^2 \Vert D^2  f_{nm} \Vert_{0}
\leq  C_T M \vert y\vert^2,\;\; y \in \R^d.
\]
Moreover, by  Lemma \ref{Plemma}, as $m\rightarrow \infty$ and 
$n\rightarrow \infty$,  we have that
\begin{equation}
\label{new5}
\langle y ,  DP_sf_{nm}(x)\rangle  = \langle e^{sA}y , P_s Df_{nm}(x) \rangle 
\rightarrow  \langle e^{sA}y , P_s Df(x)  \rangle =\langle y , DP_sf(x) \rangle,
\end{equation}
$y \in \R^d. $ Therefore, as $m\rightarrow \infty$ and $n\rightarrow \infty$,
 we find
\begin{align*}
\lim_{n \to \infty } (\lim_{m \to \infty} J_{nm}(x)) = \int_{\{\vert y \vert 
\le 1 \}}  \big(P_sf(x+y) - P_sf(x) - \langle y ,  DP_s f_{}(x)\rangle
 \big ) \nu (\mbox{d}y).
\end{align*}
Similarly, for the other terms of  $\LO_0(P_s f_{nm})$ we have that
\begin{align} \label{new6}
&\frac{1}{2}\mbox{Tr}(QD^2 P_sf_{nm}(x))= \frac{1}{2}\mbox{Tr}( Q
e^{sA^*}P_s D^2 f_{nm}(x)e^{sA}) \rightarrow\\
 \nonumber &\rightarrow\frac{1}{2}\mbox{Tr}( 
Qe^{sA^*}P_s D^2 f(x)e^{sA})= \frac{1}{2}\mbox{Tr}(QD^2 P_sf(x)),
\end{align}
 as $m\rightarrow \infty$ and 
$n\rightarrow \infty$, and
\begin{equation}
 \label{new7}
 \langle a , DP_s f_{nm}(x) \rangle =  \langle e^{sA}a , P_s Df_{nm}(x) 
\rangle \rightarrow \langle e^{sA}a , P_s Df(x) \rangle = \langle a , DP_s f(x) 
\rangle, \end{equation}
as $m\rightarrow \infty$ and $n\rightarrow \infty$, for
 $s\in[0,T]$.
We also observe that $\Vert \LO_1 (P_s f_{nm})\Vert_{0} \leq M_T$, for 
$s\in[0,T]$. It follows that
 \begin{equation*}
\int_0^t  \LO_1 (P_s f_{nm})(x) \mbox{d}s \rightarrow \int_0^t  \LO_1 (P_s 
f_{})(x) \mbox{d}s, \;\; t \in [0,T],
\end{equation*}
 as $m\rightarrow \infty$ and $n\rightarrow \infty$. Finally,
  for any $x \in \R^d$ and $s \in [0,T]$,
\begin{equation}
\label{cd4}
 \langle Ax ,  DP_s f_{nm}(x) \rangle = \langle 
e^{sA} Ax , P_s ( Df_{nm})(x) \rangle 
\rightarrow \langle e^{sA} A x , P_s (Df) (x) \rangle = \langle Ax ,  DP_s 
f(x) 
\rangle.
\end{equation}                                      
Therefore, for $t\in[0,T]$, as $m\rightarrow \infty$ and $n\rightarrow \infty$,
 $\int_0^t  \langle A x ,  DP_s f_{nm}(x) \rangle \mbox{d}s$ $ \rightarrow \int_0^t 
\langle A x ,  DP_s f(x) \rangle \mbox{d}s.$
In conclusion, passing to the limit as $m \to \infty$ and $n \to \infty$ in 
\eqref{Peqfnm}, we obtain \eqref{Peq5} and the proof is complete.
\end{proof}


\begin{theorem}\label{PPCAUCHY} Let $f\in C^2_b(\R^d)$.
If we set $u(t,x)=P_tf(x)$, $t \ge 0, $ $x \in \R^d$, where $(P_t)$ is given in \eqref{for1}, then $u$ is the 
unique bounded classical solution to the Cauchy problem
 \eqref{ca4}.
 \end{theorem}
\begin{proof} \textit{Existence.} By \eqref{con2}   we know that
 $u(t,x)=P_tf(x)$
is bounded and continuous on $[0, +\infty) \times \R^d$.
Moreover, differentiating under the integral sign  it is easy to see 
that there exist classical partial derivatives 
 $\partial_{x_i} u$ and 
 $\partial_{x_i x_j}^2 u$ which are bounded and continuous 
 on $[0, \infty) \times \R^d$.

 We have also to verify that  $u(\cdot, x)$ is a $C^1$-function on 
$[0, +\infty)$, for any $x \in \R^d$.
This follows easily from \eqref{Peq5} and  \eqref{dd2} using 
again \eqref{con2}. Theorem \ref{PCAUCHY}  also shows that $u$  solves \eqref{ca4}.


 \vv \textit{Uniqueness.}  Let $u$ be a bounded classical  solution to  \eqref{ca4}. To prove uniqueness we will use  a quite standard probabilistic argument based on  the It\^o  formula  (cf.  
\cite[Section 4.4]{Applebaum}).  First we recall the
  L\'evy-It\^o decomposition formula (see \cite[Chapter 4]{Sato}
 or 
\cite[Chapter 2]{Applebaum}).
 According to  \eqref{ft1} this formula says that on the fixed stochastic basis there exist a $Q$-Wiener 
 process $W^Q = (W^Q_t)$
with covariance
matrix  $Q$   and 
an independent Poisson random measure $N$ on $\R_+ \times ( \R^d \setminus \{0\})$ 
with
intensity measure $l \otimes \nu$ (here $l$ is Lebesgue measure on $\R_+$) 
such that
\begin{equation} \label{levyito}
Z_t = at + W_t^Q + 
\int_0^t \int_{\{ |x| \le 1\} } x 
\tilde N(ds, dx) 
+
 \int_0^t \int_{\{ |x| > 1 \} } z  N(ds, dx),
\end{equation}
$t \ge 0$; here $\tilde N$ is the compensated Poisson measure (i.e., $\tilde N 
(dt,  dx)$ $ = N(dt, dx)-  dt \nu(dx)$).
We apply   the It\^o  formula  to  the   Ornstein-Uhlenbeck 
process $(X_t^x) $. 
  We fix $t >0$ and  define, for $s \in [0,t]$, $x \in \R^d,$ 
$v(s,x)= u(t-s,x)$. 
  We have 
 \begin{align*}
& v(t,X_t^x) - v(0,x) = f(X_t^x) - u(t,x) 
\\
& = \int_0^t \int_{ \R^d \setminus \{ 0\}} [ u(t-r,X_{r-}^x + x) -
u(t-r, X_{r-}^x)]
   \, \tilde N(dr, dx) 
\\ &
+  \int_0^t \langle D u(t-r , X_r^x), dW_r^Q\rangle   
+ \int_0^t \big (-\partial_s u(t-r , X_r^x) + \LO_0 u(t-r,X_r^x)\big) dr.
 \end{align*}
Since the last integral is zero, by taking 
the expectation, we get
 $
  \E[f(X_t^x)] = u(t,x),
 $ 
and, therefore, $u(t,x) = P_t f(x)$, $x \in \R^d$.
\end{proof}

\subsection{An invariant  core  in $C_0(\R^d)$ 
}\label{core problem}
In this section    we study the  Ornstein-Uhlenbeck 
semigroup $(P_t)$ acting on $C_0(\R^d)$. This is a strongly continuous semigroup or a $C_0$-semigroup of contractions (cf. \eqref{strong1}). 
We  determine a  core $\D_0$ which is invariant for the semigroup.  

Let us denote by $\LO$ the {\it generator} of $(P_t)$, i.e., 
$ \LO f$ $=\lim_{t\to 0}\frac{P_t f - f}{t},$
with domain  ${D}(\LO)$ being the set of all $f \in C_0(\R^d)$ such that the previous limit 
exists as a limit
 in $ C_0(\R^d)$
(see, for instance, \cite{EN} for the theory of linear $C_0$-semigroups). Let us recall the   general definition of core for a $C_0$-semigroup (cf. \cite[Section 31]{Sato}).

Let $(S_t)$ be a $C_0$-semigroup on a Banach space $X$ with generator $\cal A : D({\cal A})$ $ \subset X$ $\to X$.
A \textit{core}  for $(S_t)$ (or for $\cal A$)  is a subspace 
  $\D \subset D({\cal A})$ such that
 for every $\psi \in D({\cal A})$ there exists a sequence $(\psi_n) \subset \D$ which verifies: $\psi_n \to \psi$ and  ${\cal A} \psi_n  \to
  {\cal A} \psi$ in $X$.
%

Using also Theorem \ref{PCAUCHY}  we show that the  space 
$$\D_0 = \{ f \in C^2_0 (\R^d) \; : \;x \mapsto \langle Ax , Df(x) \rangle \in 
C_0(\R^d)  \}
$$ 
is an invariant  core for $(P_t)$.
\begin{theorem} \label{cor1} Let ${\cal L}$ be the generator of the OU semigroup $(P_t)$ in $C_0(\R^d)$. The following statements hold:

\vv (i)  $\D_0 \subset D({\cal L})$ and ${\cal L} f = {\cal L}_0 f$, 
for any $f \in \D_0$ (${\cal L}_0$ is defined in (\ref{ouin}));

\vv (ii) $\D_0$ is invariant for $(P_t)$, i.e., $P_t (\D_0) \subset \D_0$, $t 
\ge 0$;
\vv (iii)  $\D_0$ is a {core} for $\LO$.
\end{theorem}
\begin{proof} {\bf (i)} Let $f \in {\D}_0$. 
We will use It\^o's 
formula (see   \cite[Section 4.4]{Applebaum}). Since in 
particular  $f \in C^2_b(\R^d)$ we have  $\P$-a.s. (cf. \eqref{levyito})
\begin{gather} \label{dr4}
f(X_t^x) = f(x) + \int_0^t {\cal L}_0 f(X_s^x)ds 
\\
\nonumber 
+  \int_0^t \int_{ \R^d \setminus \{ 0\}} [ f (X_{s-}^x + y) -
f(X_{s-}^x)]
   \, \tilde N(ds, dy)
 + \int_0^t   \langle D f  (X_{s}^x)  , dW^Q_s \rangle, \;\; x \in \R^d,
\end{gather}
$t \ge 0$. Since ${\cal L}_0 f$ is a bounded function by taking the  
expectation we get (recall that  $f$ and $Df$ are bounded and so the stochastic 
integrals have both  mean zero)
\begin{equation} \label{f5}
 P_t f(x)= f(x) + \int_0^t P_s \LO_0  f  (x) \mbox{d}s,\;\; x \in \R^d.
\end{equation}
We can write for $t>0$, $x \in \R^d,$
\begin {equation} \label{tre}
 \frac{P_t f (x) - f(x)}{t} - {\cal L}_0 f(x) =
  \frac{1}{t}\int_0^t [ P_s \LO_0  f  (x) - {\cal L}_0 f(x)]  \mbox{d}s.
 \end{equation}
From this formula it is easy to deduce that $f \in D({\cal L})$ and also that 
${\cal L } f  = {\cal L}_0 f$. 

\vv{\bf (ii)} Differentiating under the integral sign one checks that  $P_t (C^2_0 (\R^d)) \subset C^2_0 (\R^d)$, $t \ge 0$. Thus 
to prove that $\D_0$ is invariant for the semigroup, it is enough to show that 
for $f \in \D_0$, $t \ge 0$, we have that
\begin{equation}
\label{inv4}
 x \mapsto  \langle  Ax , DP_t f(x)\rangle \in C_0 (\R^d). 
\end{equation} 
To check  \eqref{inv4} we use
Theorem \ref{PCAUCHY} and \eqref{f5}. 
 We know that  for  $x \in \R^d,$ $t 
\ge 0,$
$$
 \int_0^t P_s \LO_0  f  (x) \mbox{d}s = \int_0^t \LO_0 P_s    f  (x) 
\mbox{d}s.
$$
Since $s \mapsto \LO_0 P_s    f  (x) $ and $s \mapsto P_s \LO_0     f  (x)$ are both
continuous functions we get 
$$
\LO_0 P_t    f(x) = P_t \LO_0    f(x),  \;\;  t \ge 0,\; x \in \R^d.
$$
Let us fix $t>0$. The previous identity shows in particular that   $\LO_0 P_t    f
 \in C_0 (\R^d)$.  We have (see \eqref{vi2})
$$
 \LO_0 P_t    f(x) = \LO_1 P_t    f(x) + \langle Ax, D P_t f(x) \rangle, \;\; x 
\in \R^d,
$$
 and since $P_t f \in C_0(\R^d)$ one can easily check that $\LO_1 P_t    f \in C_0 (\R^d)$ (cf. the proof of  \cite[Theorem 31.5]{Sato}).
%
It follows  
 that  $\langle A \,( \cdot), D P_t f( \cdot) \rangle$ $\in C_0 (\R^d)$ 
and this gives \eqref{inv4}.

\vv {\bf (iii)} We can use a well-known criterium   for the existence of a core (see 
\cite[Proposition II.1.7]{EN}). 
First $\D_0 $ is dense in $C_0 (\R^d)$ (to this purpose note that $C^2_K 
(\R^d)$ 
is contained in $\D_0$);  then $\D_0$ is invariant for the semigroup by (ii). 
It follows that $\D_0 $ is a core for ${\cal L}$.
This  completes the proof.
\end{proof}

\begin{corollary}\label{co1}
If $f\in \D_0$, then
\begin{align*}\label{Peq2}
P_t f(x)&= f(x) + \int_0^t P_s (\LO_0 f ) (x) \mbox{d}s
= f(x) + \int_0^t \LO_0 (P_s f ) (x) \mbox{d}s, \,\,\,\,t\geq0, x\in\R^d.
\end{align*}
\end{corollary} 
\begin{proof} From a general result of semigroup theory the  previous formula holds when ${\cal L}_0$ is replaced by the  generator ${\cal L}$ and $f \in D({\cal L})$. Using Theorem \ref{cor1} we easily obtain the assertion. 
 \end{proof}

\begin{remark}\label{remimp} 
We have mentioned formula \eqref{ya1} proved in \cite[Theorem 
3.1]{SY}. This  implies that 
$C^2_K(\R^d) \subset  D({\cal L})$ and ${\cal L}_0 f = {\cal L} f$, $f \in C^2_K(\R^d)$.
However
Theorem \ref{PPCAUCHY} can not be deduced  by \cite[Theorem 
3.1]{SY} even if we require that  $f$ belongs to 
$C^2_K(\R^d)$. This is because  the space $C_K^2(\R^d)$ is not 
invariant for  $(P_t)$. From formula \eqref{ya1} 
we  only get
\begin{equation*}\label{satoconsequence}
P_t f(x)= f(x) + \int_0^t \LO (P_s f ) (x) \mbox{d}s, 
\,\,\,\,t\geq0, \, 
x\in\R^d,\; f \in C^2_k(\R^d).
\end{equation*}
\end{remark}

\section{ Infinite dimensional Ornstein-Uhlenbeck processes  }

 We consider a  real and separable Hilbert space 
${H}$ with norm $|\cdot|$ and inner 
product $\langle \cdot, \cdot \rangle$.
We fix a  L\'evy process $Z = (Z_t)$ with values in $H$ (see \cite[Chapter 4]{PZ}). Similarly to 
the case when  $H=\R^d$, $Z$  is an
$H$-valued process defined  on
 some stochastic basis $(\Omega, {\cal F}, ({\cal F}_t)_{t \ge 0}, \P)$,
  continuous in probability,
 having stationary independent increments, c\`adl\`ag  trajectories,
  and such that $Z_0 =0$.
 One has that
 \begin{equation*}
\E [e^{i \langle Z_t , h \rangle }]  = \exp ( - t \psi (h) ), \; h
\in H, \; t \ge 0,
 \end{equation*}
where the exponent  $\psi: H \to {\mathbb C} $ is defined similarly
 to  \eqref{ft1} as
\begin{equation}
\label{infi}
 \psi(h)= \frac{1}{2}\langle Q h,h  \rangle   - i \langle a, h\rangle
- \int_{H} \Big(  e^{i \langle h,y \rangle }  - 1 - \, { i \langle h,y
\rangle} \, \mathds{1}_{\{ |y| \le 1\}} \, (y) \Big ) \nu (dy);
\end{equation}
 here 
   $Q: H \to H$ is a non-negative symmetric {\it trace-class} operator, $a
\in
 H$,
 and $\nu$
 is the L\'evy (jump) measure of $Z$  
 (i.e., $\nu$ is a $\sigma$-finite (Borel)
measure on $H$, such that  \eqref{nuu} 
holds with $\R^d$ replaced by $H$).  

Note that a L\'evy-It\^o  decomposition formula 
as \eqref{levyito} holds also in infinite dimensions:
\begin{equation} \label{levyito1}
Z_t = at + W_t^Q + 
\int_0^t \int_{\{ |x| \le 1\} } x 
\tilde N(ds, dx) 
+
 \int_0^t \int_{\{ |x| > 1 \} } x  N(ds, dx),
\end{equation}
$t \ge 0$, where $ N$ is the  Poisson random measure associated to $Z$, $a \in H$  and $W^Q = (W^Q_t)$ is a  $Q$-Wiener 
 process   with values in $H$ which is independent of $N$ 
(cf. Section 2 in \cite{Applebaumart1}).

 Let $A : D(A ) \subset H \to H$ be the generator of a $C_0$-semigroup $(e^{tA})$ $= (e^{tA})_{t \ge 0}$ on $H$. By $A^* : D(A^*) \subset H \to H$ we denote its 
adjoint operator which generates the $C_0$-semigroup $(e^{tA^*})$.

We will deal with the following generalization of the {\it Ornstein-Uhlenbeck 
process}  considered  in \eqref{2}:  
\begin{equation}\label{22}
 X^x_t  =  e^{tA} x + \int_0^t e^{(t-s)A}  dZ_s = e^{tA} x + Y_t,
  \;\; t \ge 0, \; x \in H.
\end{equation}
The stochastic integral is still a limit in probability of suitable Riemann sums  (we refer to \cite{Ch} and \cite{PZ}).
The associated  \emph{Ornstein-Uhlenbeck semigroup} $(P_t)$  is still defined   as 
\begin{equation}
\label{ouuu}
P_t f(x) = \E [f(X_t^x)] =  \int_{H} f(e^{tA} x + y) \mu_t(dy),\;\;
t \ge 0, \; x \in H, \; f \in B_b(H).
\end{equation} 
where $\mu_t$ is the law of $Y_t$ and has characteristic function
\begin{equation}\label{cd}
\hat {\mu_t} (h) 
= \E [e^{i \langle  Y_t, h\rangle}]
=
   \exp { \big(- \int_0 ^t \psi (  e^{s A^*}  h) ds \big)}, 
\;\;\; h \in H,\; t \ge 0
\end{equation}
(cf.  \cite[Corollary 1.7]{Ch} or   \cite{FR}).   
 In constrast with Section 3.2
if $H$ is infinite dimensional then the space $C_0(H)$ is in general not invariant for 
the OU semigroup 
(see page 91  in 
\cite{Applebaumart1} for more details). Hence
it is convenient to deal with $(P_t)$ acting on $C_b (H) $ or $UC_b(H)$ since both spaces 
are 
invariant  (however recall that $(P_t)$ is not strongly continuous neither on 
$C_b (H) $ nor on $UC_b(H)$ if we consider the sup-norm topology; see \cite{Ce}).



   It is important to note that \eqref{con2} and \eqref{con3}
 holds for $(P_t)$ even in this infinite dimensional setting, i.e.,  
for any $f \in C_b(H)$,  the real mapping:
 \begin{gather} \label{con22} 
(t,x) \mapsto P_tf(x) \; \text{ is continuous on } \;
 [0, +\infty) \times \R^d \; \;\text {and }
\\
\label{con33}
 \lim_{t \to 0^+ } P_t f = f, \;\;\; \text{uniformly on compact sets of $H$.}
\end{gather}
 We refer to    
\cite[Lemma 2.1]{BRS}. To this purpose note that $t \mapsto \mu_t$ is continuous with respect to the weak topology  of Borel probability measures on $H$ by \cite[Lemma VI.2.1]{Parta}. Indeed for any $h \in H$, 
  $t \mapsto \hat \mu_t(h)$ is continuous on $[0, +\infty)$ and moreover, according to page 19 in \cite{FR}, for any $T >0,$ the family $(\mu_t)_{t \in [0,T]}$ is tight. 
Results similar to \eqref{con22} and \eqref{con33} are proved in     
  \cite[Theorem 4.2]{FR} and in \cite[Theorem 4.1]{Applebaumart1}.


To study the OU semigroup it is useful to fix a notion of pointwise convergence of  functions (see 
 \cite{EK},  \cite{P}, \cite{D} and \cite{Ma}). 

A sequence $(f_n) \subset {C}_b (H) $ 
is said to be  {\it $\pi$-convergent} to
a map $f \in C_b(H)$ and we shall write \ \ 
\begin{equation} \label{pii}
 f_n \buildrel \pi \over 
\longrightarrow f
\end{equation}
 as $n \to \infty $ (or $\lim_{n \to \infty} f_n $ $ 
\buildrel \pi \over =
f$) if it converges  boundedly and pointwise,
i.e.,
 $\sup _{n \ge 1 } \sup_{x \in H} | f_n(x) | $ $ = \sup_{n \ge 1} \|f_n \|_0 <  \infty$  and
$\lim _{n \to \infty } f_n (x) \, =\, f (x),$ 
$x\in H$. 

We mention 
that a related notion of uniform convergence on compact sets can also  be used (cf. \cite{Ce} and \cite{CG}).


\section {The Cauchy problem  in infinite dimensions}


Let us introduce the infinite-dimensional OU operator ${\cal L}_0$
 associated to the OU process introduced in  Section 4: 
\begin{gather}
\nonumber
 {\cal L}_0 f(x) =   
 \langle A x ,    D f (x) \rangle +  {\cal L}_1 f(x);
\\ \label{vi22}
{\cal L}_1 f(x) = \int_{H}\left(  f(x+y) - f(x) 
-\mathds{1}_{\{\vert y \vert \leq 1\}} 
\langle y , Df(x) \rangle\right)\nu(\mbox{d}y)  
\\ \nonumber +  \frac{1}{2}\mbox{Tr}(QD^2 f(x)) + \langle a, Df(x)\rangle, 
\end{gather}
where $f \in C_b^2(H)$, $x \in D(A)$
($a$, $Q$ and $\nu$ are given  in the L\'evy-Khintchine formula \eqref{infi}). Note that by a well-known result  $QD^2f(x)$ is a trace class operator, for any $x \in H$ (see, for instance, \cite[Appendix  C]{DZ}).
We consider   an infinite dimensional Cauchy problem 
which generalizes \eqref{ca4}, i.e., 
 \begin{align} \label{cau4}
\begin{cases} \partial_t u (t,x) = \LO_0 u(t, x)
\\
u(0,x) = f(x).
\end{cases}
\end{align}
Let us introduce the space  
 \begin{equation}
\label{ca1}
C^2_A(H) = \{ f \in C^2_b(H)  \; : \; \; 
 \text{$Df(x) \in D(A^*)$, $x \in H$, and} \;  
A^* Df \in C_b(H,H) \}.
\end{equation}
A similar space has been considered in  \cite{GK} and \cite{Applebaum}. However   
in contrast with 
 to  \cite{GK} and \cite{Applebaum} we do not require 
that  the mapping $ x \mapsto \langle x, A^* Df(x) \rangle$ 
belongs to $C_b(H)$. Clearly, when $H= \R^d$ the space $C^2_{A}(H)$ coincides with $C^2_b(\R^d)$. Note that if $f \in C^2_{A}(H)$ then
\begin{equation} \label{fuy}
{\cal L}_0 f(x) =   
 \langle  x ,    A^* D f (x) \rangle +   {\cal L}_1 f(x), \;\;   x \in H.
\end{equation}

A {\it bounded classical solution $u$ } to the Cauchy problem \eqref{cau4} 
 is a 
bounded and continuous real function defined on $E = [0,+\infty) \times H$,  such that 

\vv (i) $u(t, \cdot ) \in C^2_A(H)$, $t \ge 0$,  and 
 $D_x u$, $A^* D_x u : E \to H,$ $D^2_x u : E \to L_2(H)$ are bounded and continuous functions;

\vv (ii) $u(\cdot ,x)$ is a $C^1$-function on $[0,+\infty)$, $x \in H,$ and $u$ solves \eqref{cau4}.

\medskip
To show solvability of  \eqref{cau4} we first extend
Theorem 
\ref{PCAUCHY} to infinite dimensions. 
\begin{theorem}\label{PCAUCHYi}
 Let $f\in C_b^2(H)$. Let $(P_t)$ be the OU semigroup  defined in \eqref{ouuu}. The following statements hold:

\vv (i) for any $x \in D(A)$, the real mapping $t \mapsto 
\LO_0 (P_t f ) (x)$ is continuous on $[0, + \infty)$ and $\lim_{t \to 0^+}
  \LO_0 (P_t f ) (x) = \LO_0 f (x)$. Moreover, we have:
\begin{equation}\label{Peq53}
P_t f(x)= f(x) + \int_0^t \LO_0 (P_s f ) (x) \mbox{d}s, \,\,\,\,t\geq0, \;\;
x\in D(A);
\end{equation}
(ii) if in addition $ f \in C^2_A(H)$, then  (i) and  \eqref{Peq53} hold 
for any $x \in H$.
\end{theorem}
To prove the theorem we will use results from Section 3.1. 

Recall that a   function $g: H  \to \R$ is called {\it cylindrical} if there exist 
${h_1, \ldots, h_n \in H}$ and $l : \R^n \to \R$ such that  
$$
g(x) = l (\langle x, h_1 \rangle, \ldots, \langle x, h_n \rangle ),\;\;\; x \in H.
$$
Let us fix an orthonormal basis $(e_k)$ in $H$ and consider the orthogonal projections $P_n : H \to H$,
\begin{equation}
\label{pr}
P_n x = \sum_{k=1}^n \langle x, e_k \rangle e_k,\;\;\; x \in H.
\end{equation} 
We have the following quite standard approximation result. 
\begin{lemma}
\label{uni1} Let $f \in C^2_b (H)$ and consider the cylindrical functions $(f_n) \subset C^2_b (H)$, $f_n (x) = f(P_n x)$, $x \in H$, $n \ge 1$. We have, for any $x, h \in H$, passing to the limit as $n \to \infty$,
$$
f_n(x) \to f(x),\;\;\; \langle Df_n(x), h \rangle \to 
 \langle Df(x), h\rangle, \;\;
  Df_n^2 (x) \to D^2f(x) \; \text {in $L_2(H)$};
$$$$
\;\; \text{with} \;\; \sup_{n \ge 1} \sup_{y \in H} \big(|f_n(y)| + |Df_n(y)| + \| D^2f_n(y) \|_{HS} \big) = M < \infty.
$$
\end{lemma}
\begin{proof} Since $P_n x \to x$ as $n \to \infty$ and $|P_n x| \le |x|$, $n \ge 1$, $x \in H$, it is easy to prove the assertions about $f_n  $ and $Df_n$. Let us only consider $D^2f_n$ and fix $x \in H$. We have
 $D^2 f_n(x)$   $ =  P_n D^2 f(P_n x) P_n $ and 
$$
\| P_n D^2 f(P_n x) P_n - D^2 f(x) \|^2_{HS}
$$$$
 \le   2 \| P_n [D^2 f(P_n x) - D^2f(x)] P_n  \|^2_{HS} 
+ 2\| P_n D^2 f(x) P_n - D^2 f(x) \|^2_{HS}
$$
$$
\le  2 \| D^2 f(P_n x) - D^2f(x)  \|^2_{HS} + 
2 \sum_{j, k >n}   \langle D^2 f(x) e_j , e_k \rangle^2.
$$
Using also the continuity of $D^2f: H \to L_2(H)$ we find that 
 $$\| P_n D^2 f(P_n x) P_n  - D^2 f(x) \|^2_{HS}  \; 
\to 0 \;\;  \text{ as $n \to \infty $.} 
$$
To prove the uniform estimate on $\| D^2 f_n \|_0$  we  note that 
$\| P_n D^2 f(P_n x) P_n \|_{HS}$ $ \le  \sup_{y \in H}
\|  D^2 f(y)\|_{HS}$, for any $x \in H,$ $n \ge 1$.
\end{proof}

\begin{proof}[Proof of Theorem \ref{PCAUCHYi}] We  follow the method of the proof of Theorem \ref{PCAUCHY}. We only indicate some changes.

First the 
assertion about the continuity of $t \mapsto  
\LO_0 (P_t f ) (x)$ can be proved by  an  identity  like 
\eqref{dd2} with $x \in D(A)$ and $\R^d$ replaced by $H$.

\vv {\bf (i)} We split the proof of (i) into three parts.

\vv \text{\it I Step.} We prove the assertion \eqref{Peq53} when $f$ is
a cylindrical function of the form 
$$
f(x) = 
  l (\langle x, e_1 \rangle, \ldots, \langle x, e_n \rangle ),\;\;\; x \in  H, \;\;\; \text{$l \in V(\R^n)$},
$$
where  $n \ge 1$ and $(e_k)$ is an orthonormal basis in $H$. It is enough  to consider  
  $l (x_1, \ldots, x_n)$ $= $ $
e^{i ( h_1 x_1 + \ldots +    h_n x_n ) },
$ for a fixed $h= (h_1, \ldots, h_n)\in\R^n$. Let us define $ S h = 
 \sum_{j=1}^n h_j e_j \in H$. Note that 
$$
f(y)  = e^{i \,  \langle y, Sh \rangle },\;\; y \in H
$$
(we are  using  the inner product in $H$).
 Let  $x\in D(A)$; we have  
\begin{align*}
&\frac{\partial}{\partial t} (P_t f)(x) \bigg\vert_{t=0} 
=\frac{\partial}{\partial t} (\E[ e^{i \; \langle \, S h , e^{tA}x + \int_0^t e^{(t-s)A} 
dZ_s \rangle}])\bigg\vert_{t=0}=\\
&=i \langle S h , e^{tA} Ax \rangle e^{i\langle Sh , e^{tA}x \rangle} 
e^{-\int_0^t 
\psi (e^{sA^*} Sh) ds } - e^{i \langle Sh ,  e^{tA} x \rangle } \psi ( e^{tA^*} Sh) 
 e^{-\int_0^t \psi ( e^{sA^*} Sh) ds  } \bigg\vert_{t=0} =\\
&=  \langle  D f(x) , A x \rangle - e^{i \langle Sh , x \rangle} \psi(Sh),
\end{align*}
where $\psi$ is the exponent of the L\'evy process $(Z_t)_{t\geq0}$ (see \eqref{infi}).  
Using the L\'evy-Khintchine formula, we find that   
 \begin{equation}
\frac{\partial}{\partial t} (P_t f)(x) \bigg\vert_{t=0} = \LO_0 f(x),\;\; x \in D(A). 
\end{equation}
Similarly, recalling that $Y_t=\int_0^t e^{(t-s)A} \mbox{d}Z_s$, we can compute 
the derivative for $t>0$:
 $
\frac{\partial}{\partial t} (P_t f)(x)$ $ = \LO_0 (P_t f)(x), $ $ t\geq0. 
$
Integrating with respect to $t$, we prove \eqref{Peq53}  for our cylindrical function  $f$.

\vv \text{\it II Step.} We prove the assertion  when $f : H \to \R$
 is cylindrical of the form
$$
f(x) = 
  g (\langle x, e_1 \rangle, \ldots, \langle x, e_N \rangle ),\;\;\; x \in  H, \;\;\; \text{$g \in C^2_b(\R^N)$},
$$
for some  $N \ge 1$.
If $R_N : H \to \R^{N}$, $R_N x = 
 (\langle x, e_1 \rangle, \ldots, \langle x, e_N \rangle )$ then $f(x) = g(R_N x)$, $x \in H.$

According to Lemma \ref{Plemma}
we can  choose 
$(f_{nm})_{n,m\in\N} \subset V (\R^N)$  to approximate  $g$.
Let $F_{nm} : H \to \R$, $F_{nm} (x) = f_{nm}(R_N x) $ and similarly 
 $F_{n} (x) = f_{n}(R_N x) $, $x \in H$. 
By the previous step we can write 
\begin{equation}\label{Peqfnm1}
P_t F_{nm}(x)= F_{nm}(x) + \int_0^t \LO_0 (P_s F_{nm}) (x) \mbox{d}s, 
\,\,\,\,t\geq0, x\in D(A), 
\end{equation}
for any $n,m\in\N$. 
Note that,
for any $n \ge 1$,
\begin{gather}
  F_{nm} \buildrel \pi \over 
\longrightarrow F_n,\;\;\;\;  \langle DF_{nm}(\cdot), h \rangle \buildrel \pi \over 
\longrightarrow \langle DF_n(\cdot), h \rangle,
\\ \nonumber \langle D^2 F_{nm}(\cdot)h, k \rangle \buildrel \pi \over 
\longrightarrow \langle D^2 F_n(\cdot)h, k \rangle \;\;\; 
\text{as} \; m \to \infty;
\\  \nonumber  F_{n} \buildrel \pi \over 
\longrightarrow F,\;\;\;\;  \langle DF_{n}(\cdot), h \rangle \buildrel \pi \over 
\longrightarrow \langle DF(\cdot), h \rangle,
\\ \nonumber\langle D^2 F_{n}(\cdot)h, k \rangle \buildrel \pi \over 
\longrightarrow \langle D^2 F(\cdot)h, k \rangle \;\;\; 
\text{as} \; n \to \infty,
\end{gather}
 $h, k \in H,$ and 
$$
\sup_{y \in H} \big ( |F_{nm}(y)| + |F_n(y)| +
 |DF_{nm}(y)| + |DF_n(y)| + 
  \| D^2F_{nm}(y) \|_L + \|D^2F_n(y) \|_L
  \big) = C 
$$
where $C $ is a constant independent of $n $ and $m$. 
Thanks to the previous formulas in order to pass to the limit in \eqref{Peqfnm1} 
we fix $T>0$, $x \in D(A)$, and study the convergence of $\LO_0(P_s F_{nm})(x)$, 
with $s \in [0,T]$.
The term 
$$
\int_{H} (P_s F_{nm}(x+y) - P_sF_{nm}(x) - \mathds{1}_{\{ |y|\le 1 \}}\langle 
y , DP_s F_{nm}(x)\rangle ) \nu (\mbox{d}y)
$$
can be splitted  as in \eqref{new2};  passing to the limit first as $m \to \infty$ and then as $n \to \infty$, 
we get 
$$
\int_{H} (P_s F_{}(x+y) - P_s F_{}(x) - 
\mathds{1}_{\{ |y|\le 1 \}}\langle 
y , DP_s F_{}(x)\rangle ) \nu (\mbox{d}y).
$$ 
Similarly,   it follows that 
$\int_0^t  \LO_1 (P_s F_{nm})(x) \mbox{d}s $ $\rightarrow \int_0^t  \LO_1 (P_s 
F_{})(x) \mbox{d}s,$ $ t \in [0,T],$ 
first  as $m\rightarrow \infty$ and then as $n\rightarrow \infty$. Finally as in \eqref{cd4} we find,  for $s \in [0,T]$,
$$
 \langle Ax ,  DP_s F_{nm}(x) \rangle = \langle 
e^{sA} Ax , P_s ( DF_{nm})(x) \rangle 
\rightarrow 
\langle Ax ,  DP_s 
F(x) 
\rangle,
$$
first  as $m\rightarrow \infty$ and then as $n\rightarrow \infty$. 

Moreover, 
$
| \langle Ax ,  DP_s F_{nm}(x) \rangle | $ $
+ | \langle Ax ,  DP_s F_{n}(x) \rangle |$ $ \le c_T |Ax|,
$
$ n, \,  m \ge 1$. 
Therefore, for $t\in[0,T]$, as $m\rightarrow \infty$ and $n\rightarrow \infty$,
\begin{equation*}
\int_0^t  \langle A x ,  DP_s F_{nm}(x) \rangle \mbox{d}s \rightarrow \int_0^t 
\langle A x ,  DP_s F(x) \rangle \mbox{d}s.
\end{equation*}
In conclusion,  as $m \to \infty$ and $n \to \infty$ in 
\eqref{Peqfnm1}, we obtain the assertion.

\vv \text{\it III Step.} We prove \eqref{Peq53} when $f \in C^2_b(H).$

We consider a sequence of cylindrical functions $(f_n)$ which approximate $f$ as in Lemma \ref{uni1}. By the previous step we know that \eqref{Peq53} holds when $f$ is replaced by $f_n$, $n \ge 1$, i.e., 
$$
P_t f_n(x)= f_n(x) + \int_0^t \LO_0 (P_s f_n ) (x) \mbox{d}s, \,\,\,\,t\geq0, \;\;
x\in D(A).
$$
In order to pass to the limit as $n \to \infty$ we  proceed as in the previous step. The only difficulty concerns  the term 
$\frac{1}{2}\mbox{Tr}(QD^2 P_sf_{n}(x))$.
We have  to justify the following limit 
\begin{align} \label{new66}
&\frac{1}{2}\mbox{Tr}(QD^2 P_sf_{n}(x))= \frac{1}{2}\mbox{Tr}( Q
e^{sA^*}P_s D^2 f_{n}(x)e^{sA}) \rightarrow\\
 \nonumber &\rightarrow\frac{1}{2}\mbox{Tr}( 
Qe^{sA^*}P_s D^2 f(x)e^{sA})= \frac{1}{2}\mbox{Tr}(QD^2 P_sf(x)),
\end{align}
 as $n\rightarrow \infty$. To this purpose we use basic properties of 
trace class operators
(cf.  Appendix C in \cite{DZ})
$$
| \mbox{Tr}( Q 
e^{sA^*}P_s [ D^2 f_{n}(x) - D^2 f_{}(x)] e^{sA})|
= |\mbox{Tr}(e^{sA} Q 
e^{sA^*}P_s [ D^2 f_{n}(x) - D^2 f_{}(x)] )|
$$
$$
\le \|e^{sA} Q 
e^{sA^*} \|_{HS} \, \|  D^2 f_{n}(x) - D^2 f_{}(x) \|_{HS}
$$
which tends to 0 as $n \to \infty$. Note that we also have  the bound
$$
|\mbox{Tr}(QD^2 P_s f_{n}(x))| \le  |\mbox{Tr}( Q
e^{sA^*}P_s D^2 f_{n}(x)e^{sA})| \le C_T \sup_{y \in H}\| D^2 f_n(y)\|_{HS} \le M_T,
$$
 for any $t \in [0,T]$, $n \ge 1$.
The assertion (i) is proved.

\vv {\bf (ii)}
 Let us fix $x_0 \in H$. There exists a sequence $(x_n ) \subset D(A)$ such that $x_n \to x_0$.

According to (i) and \eqref{vi22}, we have, for any $n \ge 1$, $t \ge 0$, 
\begin{gather*}
P_t f(x_n) - f(x_n) =  \int_0^t {\cal L}_1 (P_s f ) (x_n) \mbox{d}s
+  \int_0^t   \langle A x_n ,  DP_s f(x_n) \rangle  \mbox{d}s
\\
=  \int_0^t {\cal L}_1 (P_s f ) (x_n) \mbox{d}s
+  \int_0^t   \langle  e^{sA} x_n ,   P_s (A^* Df) (x_n) \rangle  \mbox{d}s. 
\end{gather*} 
Since  $A^* Df : H \to H$ is bounded and continuous, the mapping:
$x \mapsto \LO_0 (P_t f ) (x)$ is  continuous, for any $t \ge 0$.  
 Hence  we can pass to  the limit as $n \to \infty$ in the previous identity  and get 
$$
P_t f(x_0) - f(x_0) =    \int_0^t {\cal L}_1 (P_s f ) (x_0) \mbox{d}s
+  \int_0^t   \langle  e^{sA} x_0 ,   P_s (A^* Df) (x_0) \rangle  \mbox{d}s. 
$$ 
This shows  that 
\eqref{Peq53} holds  for $x = x_0$.  
\end{proof}
To prove uniqueness 
for the Cauchy problem we need 
   It\^o's formula for OU processes and functions $f \in C^2_A(H)$.
 A special case of this formula which is enough for our  purposes   follows from  \cite{SZ}. 
\begin{lemma}
\label{it3}
Let us consider an OU process $(X_t^x)$ as in \eqref{22} and assume that the L\'evy measure $\nu$ has bounded support in $H$. Let $f \in C^2_A(H)$, we have (cf. (\ref{vi22}))
$$
f(X_t^x) - f(x) = \int_0^t \int_{ H} [ f(X_{s-}^x + y) -
f(X_{s-}^x)]
   \, \tilde N(ds, dy) 
+  \int_0^t \langle D f (X_r^x), dW_r^Q\rangle $$$$ 
 + \int_0^t   \LO_1 f(X_r^x)  dr
  +   
  \int_0^t \langle A^* D f(X_r^x), X_r^x \rangle dr,\;\;\; t \ge 0, \; x \in H.   
$$
\end{lemma}
\begin{proof} The assertion can be deduced by    \cite[Lemma 7.4]{SZ} using the decomposition formula 
\eqref{levyito1}.  Note that it  is a generalization of the finite-dimensional It\^o's formula proved  in 
  \cite[Section 4.4]{Applebaum}.
\end{proof}
 
\begin{remark} {\rm    The assumption that $\nu$ has bounded support can be removed in Lemma \ref{it3}. The proof of  this more general It\^o's formula is not difficult but it requires many details (the basic idea is  to   use  the Yosida approximations of $A$ as in 
the proof of  \cite[Lemma 7.4]{SZ}).} 
\end{remark}


Next we obtain well-posedness for the Cauchy problem \eqref{cau4}.  

\begin{theorem}\label{PPCAUCHY2} Let us consider the Cauchy problem \eqref{cau4} with $f\in C^2_A(H)$.
If we set $u(t,x)=P_tf(x)$, $t \ge 0,$ $x \in H,$ where $(P_t)$ is defined  in \eqref{ouuu}, then $u$ is the 
unique bounded classical solution to 
 \eqref{cau4}.
\end{theorem}
\begin{proof}  \textit{Existence.} We know that
 $u(t,x)=P_tf(x)$  
is bounded and continuous on $[0, +\infty) \times H$ (see \eqref{con22}). 
  Moreover, differentiating under the integral sign,  it is straightforward to check that $u(t, \cdot ) \in C^2_A(H)$, $t \ge 0$, and 
$D_x u$, $A^* D_x u,$ $D^2_x u$ are bounded and continuous functions on  $[0,+\infty) \times H$.

From Theorem \ref{PCAUCHYi} we deduce that 
$u(\cdot, x)$ is a $C^1$-function on 
$[0, +\infty)$, for any $x \in H$, and 
finally that $u$  solves \eqref{cau4}.

\vv {\it Uniqueness.} We use It\^o's formula 
 as in the proof of 
Theorem \ref{PPCAUCHY}.

For any $n \ge 2$ we define  L\'evy processes $U^n = (U_t^n)$,
 and $Z^n = (Z_t^n)$, where $Z_t^n  = Z_t - U^n_t$ and 
$$
U_t^n = \int_0^t ds\int_{\{ |x| > n \} } x  N(ds, dx), \;\; t \ge 0
$$ 
(cf. \eqref{levyito1}). Now $U^n$ and $Z^n$ are independent L\'evy processes according to \cite[Lemma 4.24]{PZ} (moreover, $U^n$ is a compound Poisson process). By \cite[Lemma 4.25]{PZ} it is straightforward to prove that 
\begin{equation}
\label{act}
\E [e^{i \langle Z_t^n , h \rangle }]  = \exp ( - t \psi_n (h) ), \; h
\in H,
\end{equation} 
where  $\psi_n : H \to {\mathbb C} $ is given as follows:
\begin{equation*}
\label{infi1}
 \psi_n(h)= \frac{1}{2}\langle Q h,h  \rangle   - i \langle a, h\rangle
- \int_{ H } \Big(  e^{i \langle h,y \rangle }  - 1 - \, { i \langle h,y
\rangle} \, \mathds{1}_{\{ |y| \le 1\}} \, (y) \Big )\,   \nu_n (dy), 
\end{equation*}
 $h \in H$  (cf. \eqref{infi}); here  $\nu_n (A) = \nu (A \cap  \{ |y| \le n \} )$, 
 for any Borel set $A \subset H$  (i.e.,  $\nu_n (dy) = 
 \mathds{1}_{\{ |y| \le n\}}(y) \,   \nu (dy)$). In particular, $\nu_n$ has bounded support.
%

Let us fix $x \in H$ and $t>0$. According to page 84 in \cite{Applebaum} or  \cite[Section 1]{Ch} 
 the OU process  driven by $Z^n$ is given by   
 \begin {gather} 
 \label{d11}
 X^n_t  =  e^{tA} x + \int_0^t e^{(t-s)A}  dZ_s^n = e^{tA} x + Y_t 
 - \int_0^t e^{(t-s)A}  dU_s^n
\\ \nonumber
= X^x_t -  \int_0^t ds\int_{\{ |x| > n \} } e^{(t-s)A}x  N(ds, dx) 
 = X_t^x - \sum_{0 < s \le t } e^{(t-s)A} (\triangle Z_s) \, \mathds{1}_{\{ |\triangle Z_s| >n \}},
\end{gather}
 where $\triangle Z_s = Z_s - Z_{s-}$ ($Z_{s-}$ is the left limit in $s$) and the last term is a finite random sum. Clearly, for any $\omega \in \Omega$, $\big (\int_0^t e^{(t-s)A}  dU_s^n \big) (\omega)$ $\to 0$ as $n \to \infty$. It follows that 
\begin{equation}
\label{din}
X_t^n \to X_t^x \;\; \text{as $n \to \infty, \;\; $     $\P$- a.s.} 
\end{equation}
%
 Now
we apply    It\^o's  formula 
 as in Lemma \ref{it3} 
to $v (s, X_s^n)$, $s \in [0,t]$,
 where
  $v(s,x)= u(t-s,x)$,  $s \in [0,t]$, $x \in H$ (in the sequel we denote 
by  $\tilde N_n$  the compensated Poisson random measure of $Z^n$). We find
 \begin{align*}
& v(t,X_t^n) - v(0,x) = f(X_t^n) - u(t,x) 
\\
& = \int_0^t \int_{H}
 [ u(t-r,X_{r-}^n + y) -
u(t-r, X_{r-}^n)]
   \, \tilde N_n(dr, dy) 
\\ &
+  \int_0^t \langle D u(t-r , X_r^n), dW_r^Q\rangle   + \int_0^t \big (-\partial_s u(t-r , X_r^n) + \LO_0 u(t-r,X_r^n)\big) dr
\\
&  - \int_0^t dr \int_{ \{ |y|  >n \} } [ u (t-r, X_{r}^n + y) - u(t-r, X_{r}^n)]
   \nu (dy)
 \end{align*}
(recall  that  the L\'evy measure of $Z^n$ is $\nu_n (dy) =  \mathds{1}_{\{ |y| \le n\}}(y) \,   \nu (dy)$).
Since $\partial_s u -  \LO_0 u =0$, by taking the
expectation, we arrive at
 \begin{equation} \label{f77}
  \E[f(X_t^n)] = u(t,x) - \E \int_0^t dr \int_{ \{ |y|  >n \} } [ u (t-r, X_{r}^n + y) - u(t-r, X_{r}^n)]
   \nu (dy).
 \end{equation}
Note that
$$ \Big |\int_0^t dr \int_{ \{ |y|  >n \} } [ u (t-r, X_{r}^n + y) - u(t-r, X_{r}^n)]
   \nu (dy) \Big| \le  2  t  \| u\|_0 \nu ( \{|y| >n\})
$$ which tends to $0$ as $n \to \infty$. Passing to the limit in \eqref{f77}
we  get 
$
u(t,x) = $ $ \E [f(X_t^x)] =  P_t f(x)$ (see also \eqref{din}). 
This proves the  uniqueness.
\end{proof}

\subsection{The Ornstein-Uhlenbeck generator ${\cal L}$ 
in 
$C_b(H)$}

The contraction OU semigroup $(P_t)$ acting on $C_b(H)$ (i.e., $P_t : C_b (H) \to C_b (H)$, $t \ge 0$, and $\| P_t f  \|_0 \le  \| f\|_0$, $t \ge 0$, $f \in C_b(H)$) preserves the 
$\pi$-convergence  (i.e., for any $t \ge 0$, $ f_n 
\buildrel \pi \over 
\longrightarrow f$ implies $P_t f_n \buildrel \pi \over 
\longrightarrow P_t f$) and 
further, for any
 $f \in C_b(H)$, $x \in H,$ the mapping: 
$
t \mapsto P_t f(x)
$ 
is continuous from $[0, 
+\infty)$ into $\R$. 

Thus $(P_t)$ belongs to the class of  {\it $\pi$-semigroups of 
contractions} on 
$C_b(H)$ considered in   \cite{P}. Note that \cite{P}   mainly deals with $\pi$-semigroups 
on  $UC_b(H)$; however according to  \cite[Section 5]{P} all the results in \cite{P} can be easily extended to 
$\pi$-semigroups on  
$C_b(H)$.  
Clearly $(P_t)$ is also a 
 {\it stochastically 
continuous Markov semigroup} according to the definition given in \cite{Ma}. Indeed 
the OU semigroup is given by kernels of probability measures.
 
 As in  \cite{P}  and \cite{Ma}
we  define the {\it   generator}  ${\cal L} :$  $D({\cal L})$ 
$\subset \, { C}_b^{}(H)$ $\to $
   ${ C}_b^{}(H)$ for $(P_t)$ (we set  $\Delta _h  =$ $\frac {P_h \, - 
I}{h}$):
\begin{equation} \label{gene}
\left \{
\begin {array} {l}
\displaystyle {  D  ({\cal L})\, =\, \{ f \in  {C}_b (H)
\; \mbox {\rm such that $\exists g \in { C}_b (H)$  },\;\;
 \lim_{h \to 0^+} \frac {P_h f(x)\, - 
f(x)}{h} = g(x),}
\\
\text{for any $x \in H$, and } \;\; \sup_{h>0} \| P_h f \, - 
f\|_0 \, {h^{-1}} < \infty \};
 \\ 
\displaystyle {    {\cal L}f(x)  = \,  \lim_{h \to 0^+} \Delta_h
f(x) = g(x),\;\;\;  f \in D({\cal L}),\; x \in H.
} 
\end {array}
\right.
\end{equation}
\begin{remark}
\label{d33}
{  One  can give   equivalent definitions for the generator of the OU semigroup. According to  \cite[Proposition 3.6]{P} the operator  ${\cal L}$ in \eqref{gene} coincides with the generator considered in 
\cite{Ce} and \cite{CG} and 
 defined by the Laplace transform of the semigroup.
Moreover,  by \eqref{con33}  one can apply  \cite[Theorem 1.1]{P} and obtain 
\begin{equation*} 
\left \{
\begin {array} {l}
\displaystyle {  D  ({\cal L})\, =\, \{ f \in  {C}_b (H)
\; : \; \; \mbox {\rm  $\exists g \in { C}_b (H)$  },\;\;
 \lim_{h \to 0^+} \sup_{x \in K}\Big | \frac {P_h f(x)\, - 
f(x)}{h} - g(x) \Big|=0,}
\\
\text{for any compact set $K \subset  H$, and } \;\; \sup_{h>0} \| P_h f \, - 
f\|_0 \, {h^{-1}} < \infty \};
 \\ 
\displaystyle {    {\cal L}f(x)  = \,  \lim_{h \to 0^+} \Delta_h
f(x) = g(x),\;\;\;  f \in D({\cal L}),\; x \in H.
} 
\end {array}
\right.
\end{equation*}
 This shows that  ${\cal L} : D({\cal L}) \subset C_b(H) \to C_b(H)$ coincides with the generator considered in  \cite{GK}, \cite{Ku} and \cite{Applebaumart1}. In particular  \cite{GK} and \cite{Applebaumart1} use the  mixed topology  $\tau$ on $C_b(H)$. This is the finest locally convex topology on $C_b(H)$ which agrees on sup-norm bounded sets with the topology of the uniform convergence on compacts.
In \cite{Applebaumart1} several properties of 
the OU semigroup $(P_t)$ 
acting on  $(C_b (H), \tau)$ are established.
}
\end{remark}

\subsection {A $\pi$-core  for the  generator $\cal L$}

Generalizing \eqref{pii} we say that a $m$-indexed multisequence $(f_{n_1, \ldots, n_m})_{n_1, 
\ldots, n_m \in \N} $ $\subset C_b(H) $ $\pi$-converges to $f \in C_b(H) $ if
for any $i = 1, \ldots, m-1$ there exists an   $i$-indexed multisequence      
$(f_{n_1, \ldots, n_i})_{n_1, 
\ldots, n_i \in \N}$  $\subset C_b(H)$   such that, for $n_1, \ldots, n_i \in \N$, 
\begin{equation} \label{d45}
 f_{n_1, \ldots, n_{i+1}} \buildrel \pi \over \longrightarrow
f_{n_1, \ldots, n_i} \;\; \text{as} \;\; n_{i+1} \to \infty.
\end{equation}                                           
We write 
 $ \lim_{n_1 \to \infty} \ldots \lim_{n_m \to \infty} f_{n_1, \ldots, 
n_m }$ $ \buildrel \pi \over =
f$  or 
$ f_{n_1, \ldots, n_{m}} \buildrel \pi \over \longrightarrow
f$. 
 Following \cite{Ma} we say that $E \subset C_b(H)$ is {\it $\pi$-dense} if for any $f 
\in C_b(H)$ there exists an  
$m$-indexed multisequence $(f_{n_1, \ldots, n_m})_{n_1, 
\ldots, n_m \in \N}$  $\subset E$  such that 
 $ 
f_{n_1, \ldots, n_{m}} \buildrel \pi \over \longrightarrow
f.
$

Moreover a subspace $\D \subset D({\cal L})$ is a 
{\it $\pi$-core} for $\cal L$ if $\D$ is $\pi$-dense in $C_b(H)$ and, for any $f \in 
D({\cal L})$ there exists an  
$m$-indexed multisequence   ${(f_{n_1, \ldots, n_m})_{n_1 , 
\ldots, n_m \in \N} }$ $ \subset \D$  such that 
 $$ 
f_{n_1, \ldots, n_{m}} \buildrel \pi \over \longrightarrow
f , \;\;\;  {\cal L}f_{n_1, \ldots, n_{m}} \buildrel \pi \over \longrightarrow
{\cal L}f. 
$$
The next  result can be proved more generally for any   stochastically 
continuous Markov semigroup  acting on $C_b(H)$ (cf.  \cite[Proposition 2.11]{Ma}).
 It generalizes a  
 classical result in the theory of $C_0$-semigroups.
\begin{theorem} \label{coma} Let $(P_t)$ be the OU semigroup with generator ${\cal L}.$
 If a  subspace  $\D \subset D({\cal L})$ is $\pi$-dense in $C_b(H)$ and moreover
$P_t (\D) \subset \D$, $t \ge 0$, then $\D$ is a $\pi$-core for ${\cal L}$. 
\end{theorem}
As in
\cite{GK} and \cite{Applebaumart1} and similarly 
to the space 
$\D_0$, we introduce
\begin{equation}
\label{corre}
{\tilde \D_0} = \{ f \in C^2_A (H) \; : \;x \mapsto \langle x , 
A^*Df(x) \rangle \in 
C_b(H)  \}.
\end{equation} 
 Note that if $f \in \tilde \D_0$ then ${\cal L}_0 f \in C_b(H)$.
Using also Theorem \ref{PPCAUCHY2}, we  prove that ${\tilde \D_0}$
 is an invariant $\pi$-core  for the OU semigroup.
We start with  a preliminary result.
\begin{proposition} \label{cor111} Let us consider the OU generator ${\cal L} $ given in \eqref{gene} and the operator ${\cal L}_0$ defined  in \eqref{fuy}. The following statements hold:

\vv (i)   $\tilde  \D_0 \subset D({\cal L})$ and ${\cal L} f = {\cal L}_0 
f$, 
for any $f \in {\tilde \D_0}$;

\vv (ii) ${\tilde \D_0}$ is invariant for $(P_t)$, i.e., $P_t ({\tilde \D_0}) \subset 
{\tilde \D_0}$, $t 
\ge 0$.
\end{proposition}
\begin{proof} {\bf (i)} This assertion follows from  
\cite[Theorem 4.2]{Applebaumart1} (see also Remark \ref{d33}). 
Moreover, by (i) we deduce  (see \cite{Applebaumart1} or  \cite[Proposition 3.2]{P})
\begin{equation} \label{f55}
 P_t f(x)= f(x) + \int_0^t P_r \LO_0  f  (x) \mbox{d}r,\;\;\; x \in H.
\end{equation}
Alternatively, to prove (i) one can first establish  \eqref{f55} 
using Lemma \ref{it3} and arguing  as in the final part of the proof of Theorem \ref{PPCAUCHY2}. Indeed we have, for $f \in \tilde \D_0 $, $t>0$, $x \in H$,
  \begin{equation} 
\label{1244}
 \E [ f(X_t^n)] - f(x)   =   \int_0^t  \E [ \LO_0 f(X_r^n) ] dr
 - \int_0^t dr \int_{ \{ |y|  >n \} } \E [ f (X_{r}^n + y) - 
f(X_{r}^n)]
   \nu (dy)
 \end{equation}
(the process $(X_t^n)$ is defined in \eqref{d11}, $n \ge 1$). Since $\LO_0 f \in C_b(H)$ we have $\E [ \LO_0 f(X_r^n) ] $ $\to \E [ \LO_0 f(X_r^x) ]$, as $n \to \infty$, $r \in [0,t]$ (cf. \eqref{din}).  Passing to the  limit in \eqref{1244}
as $n \to \infty$ we obtain \eqref{f55}.
 Once \eqref{f55} is proved one can proceed as in the proof of  
  Theorem \ref{cor1} (see \eqref{tre}) and get (i).

\vv {\bf (ii)} Differentiating under the integral sign we find  $P_t (C^2_b (H)) \subset C^2_b (H)$, $t \ge 0$. Moreover, 
 we have easily $A^* D P_t f \in C_b(H,H)$, for $f \in \tilde \D_0$, $t \ge 0$.

Thus, 
to prove (ii),   it is enough to show that 
for $f \in \tilde \D_0$, we have that the map:
\begin{equation}
\label{inv44}
 x \mapsto  \langle  x , A^* DP_t f(x)\rangle \in C_b (H), \; \; t \ge 0. 
\end{equation} 
To check  this we use Theorem \ref{PCAUCHYi} and 
  \eqref{f55}. 
 We obtain for $f \in \tilde \D_0$, $x \in H,$ $t 
\ge 0,$
$$
 \int_0^t P_s \LO_0  f  (x) \mbox{d}s = \int_0^t \LO_0 P_s    f  (x) 
\mbox{d}s.
$$
Since $s \mapsto \LO_0 P_s    f  (x) $ and $s \mapsto P_s \LO_0     f  (x)$ are both 
continuous functions we get 
$$
\LO_0 P_t    f = P_t \LO_0    f,  \;\;  t \ge 0.
$$
Let us fix $t>0$. The previous identity shows that   $\LO_0 P_t    f
 \in C_b (H)$ since $\LO_0 f \in C_b(H)$.  
 We have (see \eqref{fuy})
$
 \LO_0 P_t    f(x) $ $ = \LO_1 P_t    f(x)$ $ + \langle x, A^* D P_t f(x) \rangle,$ $ x 
\in H.
$
Since  $\LO_1 P_t    f \in C_b (H)$ 
it follows  
 that  $\langle \cdot \, , A^* D P_t f( \cdot) \rangle$ $\in C_b (H)$, $t \ge 0$. 
\end{proof}   

We need to introduce a space $\D_1$ which  is similar to the space $I_A(H) $ used in \cite{Ma} (see also Remark 2.25 in \cite{D}). 

\noindent $\D_1$ is the linear span of the real and imaginary parts of the maps
$\phi_{a,h} : H \to {\mathbb C}$,  
\begin{equation}
x \mapsto \phi_{a,h}(x) =\int_0^a P_s \big( e^{i \langle \cdot, h \rangle} \big)(x) ds 
=\int_0^a  e^{i \langle e^{sA}  x, h\rangle}\,
 e^{- \int_0 ^s \psi (  e^{r A^*}  h) dr } ds, 
\end{equation} 
where $h \in D(A^*)$, $a>0$ (cf. \eqref{cd}).

\begin{proposition}
\label{e33} Let ${\cal L}$ be the OU generator.  The following statements hold: 

\vv (i) $\D_1 \subset \tilde \D_0 \subset D({\cal L})$ and ${\cal L} f = {\cal L}_0 f$, $f \in \D_1$;

\vv (ii) $\D_1$ is invariant for the OU semigroup $(P_t)$, i.e., $P_t ({ \D_1}) \subset 
{ \D_1}$, $t 
\ge 0$;

\vv (iii) the space $\D_1$ is $\pi$-dense in $C_b(H)$.
\end{proposition}
\begin{proof}
{\bf (i)} By Proposition \ref{cor111} we have only to prove  that $\D_1 \subset \tilde \D_0$. 
If $I_1(\phi_{a,h}) $ and $I_2 (\phi_{a,h})$ denote respectively the real and imaginary part of $\phi_{a,h}$, it is enough to show that $I_j (\phi_{a,h}) \in \D_1$, $j=1,2.$

We have first to prove that $ I_j (\phi_{a,h}) \in C^2_A (H)$ (see \eqref{corre}). 
If we fix $a>0$ and $h \in D(A^*)$ we can  compute, for $x \in H$, $k \in H,$
$$
\langle D \phi_{a,h} (x), k \rangle  = \int_0^a \langle D [ P_s \big( e^{i \langle \cdot, h \rangle} \big)](x), k\rangle  ds = i \int_0^a   \langle e^{sA^*} h, k\rangle \, P_s \big( e^{i \langle \cdot, h \rangle} \big)(x)  ds;
$$
 it follows
 that $D I_j(\phi_{a,h}) \in C_b(H,H)$, $j =1,2$. 
Moreover,  $ D I_j(\phi_{a,h} (x)) \in D(A^*) $, $x\in H$, and 
 $A^*  D \phi_{a,h} (x) $ $ = i \int_0^a    e^{sA^*} A^* h \, P_s \big( e^{i \langle \cdot, h \rangle} \big)(x)  ds$; we deduce that
 $$ A^*D I_j (\phi_{a,h}) \in C_b(H,H), \;\;\;j=1,2.$$
Since
$$
\langle D^2 \phi_{a,h} (x) k', k \rangle  =-  \int_0^a   \langle e^{sA^*} h, k\rangle \,  \langle e^{sA^*} h, k'\rangle  \, P_s \big( e^{i \langle \cdot, h \rangle} \big)(x)  ds, \;\; k,\, k' \in H,
$$
using an orthonormal basis $(e_k)$, we find that $D^2 I_j (\phi_{a,h}) (x)$
 is a Hilbert-Schmidt operator for $j=1,2$, $x \in H$. Moreover, 
 for $x,y \in H$, 
\begin{gather*}
\| D^2 I_j (\phi_{a,h}) (x) - D^2 I_j (\phi_{a,h}) (y) \|^2_{HS} \\ \le  C_a \int_0^a  |h|^2 |e^{sA} (x-y)|^2  \sum_{j,k =1}^{\infty}   (\langle e^{sA^*} h, e_k \rangle \,  \langle e^{sA^*} h, e_j\rangle )^2 
\,
ds \\ \le 
c_a |h|^2 |x-y|^2 \int_0^a | e^{sA^*} h|^4 \, ds.
\end{gather*}
Using also the previous formula we obtain that $D^2 I_j (\phi_{a,h}) : H  \to  L_2(H)$ is bounded and continuous, $j =1,2$. This shows that $ I_j (\phi_{a,h}) \in C^2_A(H)$. 

To finish the proof it remains to prove that 
 $x \mapsto  \langle x , 
A^* D I_j (\phi_{a,h}) (x) \rangle \in 
C_b(H)$, $j =1,2$. 
We have, integrating by parts,
\begin{gather*}
\langle A^* D \phi_{a,h} (x), x \rangle  = i \int_0^a   \langle e^{sA^*} A^* h, x\rangle \, P_s \big( e^{i \langle \cdot, h \rangle} \big)(x)  ds
\\
=  i \int_0^a   \langle e^{sA^*} A^* h,  x\rangle \,   e^{i \langle   x, e^{sA^*} h\rangle}\,
 e^{- \int_0 ^s \psi (  e^{r A^*}  h) dr } ds 
=
  \int_0^a    \frac{d}{ds} \big [ e^{i \langle   x, e^{sA^*} h\rangle} \big]\,
 e^{- \int_0 ^s \psi (  e^{r A^*}  h) dr } ds
\\
= e^{i \langle e^{a A}  x, h\rangle} e^{- \int_0^a \psi (  e^{r A^*}  h) dr }  -  e^{i \langle  x, h\rangle} +  
\int_0^a   e^{i \langle e^{sA}  x, h\rangle} \,
 e^{- \int_0 ^s \psi (  e^{r A^*}  h) dr }
 \psi (  e^{s A^*}  h) ds.
\end{gather*}
This shows that $x \mapsto \langle A^* D \phi_{a,h} (x), x \rangle$
is a bounded and continuous function. 
The assertion follows easily.

\vv {\bf (ii)} Let us fix $t>0$. We prove that $P_t (\D_1) \subset \D_1$. We have, for $a>0,$ $x \in H,$ $h \in D(A^*)$, by using the semigroup law, 
\begin{gather} \label{al4}
P_t \phi_{a,h}(x) = \int_0^a P_{s+t} \big( e^{i \langle \cdot, h \rangle} \big)(x) ds  
= \int_0^{a+t} P_{r} \big( e^{i \langle \cdot, h \rangle} \big)(x) dr
- \int_0^{t} P_{r} \big( e^{i \langle \cdot, h \rangle} \big)(x) dr. 
\end{gather}
 Hence $P_t  \phi_{a,h}(x) =  \phi_{a+t,h}(x)  -  \phi_{t,h}(x) $.
 This shows the assertion.

\vv {\bf (iii)} One can follow  the proof  of  \cite[Proposition 4.4]{Ma}. This is based on two facts. The first one is  that 
 the linear span of all real and imaginary parts of functions
 $e^{i \langle \cdot, h\rangle }$, $h \in D(A^*)$, is $\pi$-dense in $C_b(H)$ (see   \cite[Proposition 2.37]{D} or   \cite[Proposition 4.2]{Ma}). The second fact is that ${n} \, \phi_{\frac{1}{n}, h } (x)
$ converges to $e^{i \langle x, h\rangle }$ as $n \to \infty,$ $x \in H$, $h \in D(A^*)$.
 \end{proof}
Since by Propositions \ref{cor111} and \ref{e33} the spaces $\tilde \D_0$  and $\D_1$ both satisfy the assumptions of Proposition \ref{coma} we obtain 
\begin{corollary} \label{end} The spaces $\tilde \D_0$ and $\D_1$ are both $\pi$-cores for the OU generator $\cal L$ with $\D_1 \subset \tilde \D_0 $. They are also invariant for the OU semigroup $(P_t)$ and,  for any $f \in \tilde \D_0$, one has
$
{\cal L} f = {\cal L}_0 f.
$
\end{corollary}

\begin{remark}\label{app} In    \cite[Theorem 4.5]{GK} and  \cite[Theorem 5.2]{Applebaumart1} the authors consider the space  
 ${\cal F  C}^2_A (H) \subset C^2_b(H)$ of all
{ cylindrical 
functions} $f :  H \to \R $ such that  there exists $n \in \N$, $h_1, \ldots, h_n \in D(A^*)$ and $g 
\in C^2_b(\R^n)$ with 
\begin{equation}
\label{cil1}
f (x) = g(\langle x, 
h_1\rangle, \ldots, \langle x, 
h_n\rangle),\;\;\; x \in H, 
\end{equation} 
and moreover the map: $ x \mapsto \langle A^* D f(x), x \rangle  \in C_b(H)$.  In \cite{GK} and  in \cite{Applebaumart1} it is stated that ${\cal F  C}^2_A (H)$ is a core for the OU generator ${\cal L}$ with respect to the mixed topology  
(cf. Remark \ref{d33}); for this result \cite{GK} considers the case of  Gaussian OU processes  and  \cite{Applebaumart1} assumes that $\int_{\{ |y| >1 \}} |y|\nu(dy) < \infty$. However it seems that this result requires additional assumptions on $A.$  To this purpose we only note that  even if $g \in C^2_K(\R)$ and $h \in D(A^*)$,  then in general 
 the cylindrical function $f(x) = g(\langle x , h\rangle)$, $x \in H,$ does not belong to ${\cal F  C}^2_A (H)$. 
As in \cite{LR} a  sufficient  condition in order  
that ${\cal F  C}^2_A (H)$ is a core is that there exists in $H$ an orthonormal basis ${\cal B}$  of eigenvectors of  $A^*$;  in this case ${\cal F  C}^2_A (H)$ is a core if 
 in the previous definition  we add the condition that  the  vectors  $h_1, \ldots , h_n $ in  \eqref{cil1} belong to $\cal B$
(cf. the proof of Theorem 4.5 in \cite{GK}).
 \end{remark}

\begin{remark}\label{mehler} As  an extension of the OU semigroups one can consider the     generalised Mehler
 semigroups (see \cite{BRS}, \cite{FR}, \cite{PZ1}, \cite{W}). A generalised Mehler semigroup  $(S_t)$,  acting on ${ C}_b (H) $, is given by
  \begin{equation*}\label{st1}
 S_t f(x) = \int_H f( e^{tA}x + y  ) \mu_t (dy),\;\; t \ge 0, \,
x\in H, \, f \in { C}_b (H),
 \end{equation*}
 where $(e^{tA})$ is a ${C}_0$-semigroup on $H$, with generator
  $A$,   $\mu_t$, $t \ge 0$, is a given family of probability measures on $H$, such 
that
 $ \hat \mu_t (h)= \exp \Big(- \int_0^t
\psi (e^{sA^*} h ) ds \Big),$ $ h \in H,$ $  t \ge 0.$
 Here, $\psi: H \to C $ is a (norm) continuous, negative definite function
 with $\psi (0)=0$. Moreover, we require that  $t \mapsto \mu_t$ is continuous on $[0, + \infty)$ with respect to the weak topology  of measures (cf. Lemma 2.1 in \cite{BRS}). 

One can define a generator ${\cal L}$ for $(S_t)$ 
 with $D({\cal L}) \subset C_b(H)$ as in \eqref{gene} and an associated  subspace $\D_1$ 
 (defined with $A$ and $\psi$  given before). Arguing as in \cite[Proposition 4.4]{Ma} and using \eqref{al4} one  shows that $\D_1 \subset D(\cal L)$  and 
$$
{\cal L} \phi_{a,h}(x) =  e^{i \langle e^{aA}  x, h\rangle} e^{- \int_0^a \psi (  e^{r A^*}  h) dr }  -  e^{i \langle  x, h\rangle}, \;\; x\in H, \, h \in D(A^*), \, a>0.  
$$
  Moreover,           $\D_1$ is  an invariant $\pi$-core   for $(S_t)$. 
To prove this fact first we argue as in the proof of Proposition \ref{e33}
 and establish that $S_t (\D_1) \subset \D_1$, $t \ge 0$; then 
we apply Proposition \ref{coma} to $(S_t)$ and $\D_1$.
\end{remark}

\section{
Kolmogorov equations for measures}
Following  \cite{Ma} 
the spaces $\tilde \D_0$ and $\D_1$ (see Section  5.2)
can be used 
to characterize the marginal distributions 
of the Ornstein-Uhlenbeck process  as 
solutions to Fokker-Planck-Kolmogorov equations for measures. 
 Note that \cite{Ma} deals with   stochastically 
continuous  Markov semigroups $(R_t)$ acting on $UC_b(H)$ (i.e., $R_t : UC_b(H) \to UC_b(H)$, $t \ge 0$). However all the  results in \cite{Ma} (see in particular Theorems 1.2, 1.3 and 1.4) can be easily proved for stochastically 
continuous  Markov semigroups  acting on $C_b(H)$ with the same proofs.


To state the main result
 let  $M(H )$ be the Banach space of all finite signed Borel measures on $H$ 
endowed with the total variation norm $\|  \cdot \|_{TV}$.


 \begin{definition}  
Let $\D$ be $\tilde \D_0$  or $\D_1$ and ${\cal L}_0$ be the operator defined in \eqref{fuy}. 
Given  $\mu$ $\in   M(H )$,  a family of 
 measures $(\gamma_t )_{t \ge 0} \subset M(H)$ is called a \textit{solution to the 
measure equation}
\begin{equation}
\label{ma1}
\begin{cases}
\frac{d}{dt} \int_H f(x) \gamma_t(dx) = 
\int_H {\cal L}_0 f(x) \gamma_t(dx), \;\;\; f \in { \D}, \; t \ge 0,
\\
 \gamma_0 = \mu
\end{cases}
\end{equation} 
if we have:

\smallskip


  (i) 
 for any 
$T>0$, the real map: $t \mapsto   \|  \gamma_t\|_{TV}$ 
 belongs to $L^1(0,T)$;

(ii) for any $f \in { \D}$, the real function: $t \mapsto \int_{H} f(x) 
\gamma_t 
(dx)$ is absolutely continuous on each $[0,T]$, $T>0,$ and moreover
$$
\int_H f(x) \gamma_t(dx) - \int_H f(x) \mu(dx) = 
\int_0^t \Big (\int_H {\cal L}_0 f(x) \gamma_s(dx) \Big) ds,
\;\; t \ge 0. \qed
$$
\end{definition}
To study
 \eqref{ma1} 
one  associates by duality to  
the OU semigroup $(P_t)$ another semigroup  
 $(P_t^*)$ (see \cite[Section 3]{Ma});   $P_t^* : M(H)
 \to M(H)$, $t \ge 0$, and,
for any $\mu \in M(H)$,
$$
P_t^* \mu (B)  = \int_H P_t  (\mathds{1}_B)(x) \mu (dx), 
$$
 for  Borel set  $B \subset H$.

\begin{theorem} \label{cio} Let $\D$ be $\tilde \D_0$ or $\D_1$ (see Section 5.2). Then  for any $\mu \in  M(H )$ there exists a 
unique solution $(\gamma_t)_{t \ge 0}$ $ \subset M(H )$ to equation \eqref{ma1}
 Moreover, 
such solution is given by $(P_t^* \mu)_{t \ge 0}$.
\end{theorem}
\begin{proof}
The assertion  follows from   \cite[Theorem 1.5]{Ma} (see also 
\cite[Remark 5.1]{Ma})  using the fact that both  ${\tilde \D_0}$ and $\D_1$ are  $\pi$-cores for the OU generator ${\cal L}$ (see Corollary \ref{end}).  
%
\end{proof}

\section*{Appendix }

\noindent \textbf{Sketch of the proof of Lemma \ref{Plemma}}. Let $f \in C^2_b (\R^d)$. We proceed in some steps.

\vv \emph{I Step}. We consider a $C^{\infty}$-function $\rho : \R^d \to \R$  such that, for all $x\in \R^d$, $0\leq \rho(x) \leq 1$, 
 $\rho(x)= 1$ for $|x| \le 1$ and $\rho(x) =0 $ for $|x| \ge 3/2$.
We  define a standard sequence of mollifiers  $(\rho_n)_{}$ setting
$
\rho_n(x)=\frac{1}{c}\rho(nx)n^d,
$
where $c=\int_{\R^d} \rho(x)\mbox{d}x$. Therefore, $\int_{\R^d}\rho_n (x) \mbox{d}x =1$, $n \ge 1$.
 
We introduce  a sequence $(\tilde{f}_n)_{}$ as
$
\tilde{f}_n(x)=(f\ast \rho_n)(x)=\int_{\R^d} f(y)\rho_n(x-y)\mbox{d}y$, $x \in \R^d$.  
Differentiating under the integral sign, one proves easily that each $\tilde{f}_n$ belongs to $C^{\infty}_b(\R^d)$ (i.e., each  $\tilde{f}_n$ is bounded  
 and has  bounded derivatives of all orders). It is straightforward to check that, for any compact $K \subset \R^d$,
\begin{gather*}
 \tilde{f}_{n} \to  f, \;\; 
 D\tilde{f}_{n} \to  Df, \;\; 
 D^2\tilde{f}_{n} \to D^2f,
\end{gather*}
  uniformly on $K$ as $n \to \infty$. Moreover, for any $n \ge 1$,
\begin{equation}
\label{bouu}
 \| \tilde f_{n} \|_{0} + \| D \tilde f_{n} \|_{0} +
  \| D^2 \tilde f_{n} \|_{0} \le  \| f \|_{0} + \| D f \|_{0} +
  \| D^2f \|_{0}.
\end{equation} 

\noindent \emph{II Step.} We define $f_n^* : \R^d \to \R$, $n \ge 1$,
\[ f_n^* (x)=\tilde{f}_n(x)\rho\left(\frac{x}{n}\right), \;\; x 
\in \R^d;
\]
 each $f_n^* $ belongs to  $C^{\infty}_{b}(\R^d)$ with compact support in the $d$-dimensional cube $(-2n,2n)^d$. By \eqref{bouu} we obtain  that $(f_n^*)$, $(Df_n^*)$ and $(D^2 f_n^*)$ are uniformly bounded on $\R^d$, i.e.,
$
\sup_{n \ge 1} \, (\| f_{n}^* \|_{0} + \| D f_{n}^* \|_{0} +
  \| D^2f_{n}^* \|_{0}) < \infty.
$
In addition, as $n \rightarrow \infty$, we have $f_n^* \rightarrow f$, $Df_n^* \rightarrow Df$, $D^2f_n^* \rightarrow D^2f$ pointwise on $\R^d$.

\vv \emph{ III Step.} We define suitable extensions $(f_n)_{}$ of the functions $(f_n^*)_{}$ when these are restricted to the domain $[-2n,2n]^d$.  Let $n \ge 1$. We extend $f_n^*$ from $[-2n,2n]^d$ to $\R^d$ by   
periodicity of period $4n$ in all its variables, i.e., 
\[ 
f_n (x + 4n h  ) = f_n^*(x), \;\; x\in [-2n,2n]^d, \;\; h = (h_1, \ldots, h_d) \in {\mathbb Z}^d.
\]
 Clearly, $  \sup_{n \ge 1} \, (\| f_{n} \|_{0} $ $ + \| D f_{n} \|_{0}$ $ +  \| D^2f_{n} \|_{0}) < \infty$ and $(f_n) \subset C^{\infty}_b(\R^d)$.
 Since, for all $n \ge 1$, $f_n$ has  period $4n$ in each of its variables, it can be represented by Fourier series of the kind
\begin {equation} \label{tori} f_n(x) = \sum_{h\in\Z^d} c_h^{(n)} e^{\frac{i\pi}{2n} \langle x , h\rangle},  \;\; x \in \R^d,
\end {equation}
with 
$ c_h^{(n)} = \frac{1}{(4n)^d} \int_{[-2n,2n]^d} f_n(x)e^{-\frac{i\pi}{2n} \langle x , h\rangle} \mbox{d}x$ and 
 $\langle x, h\rangle$ $= x_1 h_1 + \ldots + x_d h_d$.
It is a standard result  that the series is uniformly  convergent on $\R^d$ (see \cite[Chapter VII]{SW}).    
We introduce, for  $n , \, m \ge 1,$
\[ f_{nm}(x)= \sum_{h\in\Z^d , \,  | h | \leq m}^{} c_h^{(n)} e^{\frac{i\pi}{2n} \langle x , h\rangle} \in V(\R^d).
\]
Differentiating under the summation in \eqref{tori}  (using the regularity properties of $f_n$) we get  that, for any $n \ge 1$,
$ \lim_{m\rightarrow\infty}  f_{nm}= f_n $, $\lim_{m\rightarrow\infty}  Df_{nm}= Df_n $ and $\lim_{m\rightarrow\infty} D^2f_{nm}= D^2f_n $, uniformly on $\R^d$. This gives  the first assertion of Lemma \ref{Plemma}. 

\noindent \emph{IV Step.} Let $x_0\in\R^d$. there exists $n_0 \in \N$ such that $|x_0| < n_0$; therefore, by Step III, $f_n  (x_0) = f_n^* (x_0)$, $n \ge n_0$. This implies by Step II that 
 $f_n (x_0) \rightarrow f(x_0)$ as $n \rightarrow \infty$.  The same happens for the derivatives, that is
$Df_n (x_0)  \rightarrow Df(x_0),$
 $D^2f_n (x_0) \rightarrow D^2f(x_0),$ 
as $n \to \infty$. We have shown that $(f_n)$ and $(f_{nm})$ verify all the  assertions.
\qed







\end{document}